\documentclass[12pt,a4paper]{article}%
\usepackage{amsmath}
\usepackage{graphicx}
\usepackage{epsfig}%
\usepackage{amsfonts}%
\usepackage{amssymb}

\usepackage{color}

\usepackage[normalem]{ulem}

\newtheorem{theorem}{Theorem}[section]
\newtheorem{corollary}[theorem]{Corollary}

\newtheorem{definition}[theorem]{Definition}

\newtheorem{lemma}[theorem]{Lemma}

\newenvironment{proof}[1][Proof]{\noindent \emph{#1.} }{\hfill \ 
\rule{0.5em}{0.5em}}

\makeatletter\@addtoreset{equation}{section}\makeatother
\makeatletter\@addtoreset{figure}{section}\makeatother
\makeatletter\@addtoreset{table}{section}\makeatother
\begin{document}

\title{Prospects of tensor-based  numerical  modeling of the collective electrostatic potential  
in  many-particle  systems}

\author{Venera Khoromskaia \thanks{Max Planck Institute for
        Mathematics in the Sciences, Leipzig;
        Max Planck Institute for Dynamics of Complex Technical Systems, 
        Magdeburg ({\tt vekh@mis.mpg.de}).}
        \and
        Boris N. Khoromskij \thanks{Max-Planck-Institute for
        Mathematics in the Sciences, Inselstr.~22-26, D-04103 Leipzig,
        Germany ({\tt bokh@mis.mpg.de}).}
        }
 
  \date{}

\maketitle

\begin{abstract}
Recently the rank-structured tensor approach suggested a progress in the numerical treatment of the 
long-range electrostatic potentials  in many-particle systems and the respective interaction 
energy and forces \cite{KhKh_CPC:13,VeBokh_NLAA:16,BKK_RS:17}.
In this paper,  we outline the prospects for tensor-based numerical modeling of the
collective electrostatic  potential on lattices and in many-particle 
systems of general type. We generalize the approach initially introduced 
for the rank-structured grid-based calculation of the collective potentials on 
3D lattices \cite{KhKh_CPC:13} to the case of many-particle systems with variable charges 
 placed on $L^{\otimes d}$ lattices and discretized on fine $n^{\otimes d}$ 
Cartesian grids  for arbitrary dimension $d$. As result, the interaction potential is represented 
in a parametric low-rank canonical format in $O(d L n)$ complexity. 
The energy is then calculated in $O(d L)$ operations.
Electrostatics in large biomolecules is modeled by using the novel range-separated (RS) tensor format
\cite{BKK_RS:17}, which maintains the long-range part of the 3D collective 
potential of  the many-body  system represented on $n\times n \times n$ grid
in a   parametric low-rank form in $O(n)$-complexity. 
We show  that  the force field can be easily recovered by using 
the already precomputed electric field in the low-rank RS format.
The RS tensor representation of the discretized Dirac delta 
\cite{BKhor_Dirac:18}  enables  the construction of 
the efficient energy preserving (conservative) regularization scheme
for solving the 3D elliptic partial differential equations with strongly singular right-hand side
arising, in particular, in  bio-sciences. 
 We conclude that the rank-structured  tensor-based approximation techniques 
 provide the  promising numerical tools  for applications to  many-body 
dynamics, protein docking and classification problems, for low-parametric
interpolation of scattered data  in data science,  
as well as in machine learning  in many dimensions.
\end{abstract}



\section{Introduction} \label{sec:Intro}   
 
Numerical modeling of  electrostatics  in the large  ensembles  of charged 
particles is considered since long as a challenging problem. 
  Computation of  the electrostatic interactions in many-body dielectric systems
like solute-solvent complexes 
 is a complicated  numerical  problem in molecular dynamics simulations 
 of large solvated biological systems,  
 in docking or folding of proteins, pattern recognition, the assembly of polymer particles
 in colloidal physics  and  many other problems  
  \cite{Holst:94,Jackson:99,CaMeTo:97,HuMcCam:1999,Kaxiras:03,LiStCaMaMe:13,Lu2008,Maday:2018}. 
  Most numerical schemes for modeling these problems are based  
  either on use of FEM/FDM discretization or 
 on application of integral formulations for solving the arising PDEs. 
  On the other hand,  the complicated many-particle interaction processes  
 are often modeled by using 
 the stochastic Monte Carlo  approaches, avoiding the traditional deterministic 
 methods of high computational complexity.

 Tensor-structured  numerical methods are now becoming popular in scientific computing due to their 
intrinsic property of reducing the grid-based solution of multidimensional problems in 
 $\mathbb{R}^d$  to basically ``one-dimensional'' computations. 
These methods evolved from bridging of the traditional 
rank-structured tensor formats in multilinear algebra 
\cite{De_Lath_PhD:97,Comon:02,Cichocki:2002,smilde-book-2004,Hack_Book12,GraKresTo:13} with the 
nonlinear approximation theory based on a separable representation of multivariate functions and 
operators \cite{HaKhtens:04I,GHK:05,Khor1:06}. 
One of the important ingredients of 
tensor  decomposition methods in multilinear algebra  was the so-called higher 
order SVD (HOSVD) introduced for the rank reduction 
 in  the Tucker tensors \cite{DMV-SIAM2:00} and further extended to the TT tensor format 
\cite{OsTy_TT:09,Gras:2010,Osel_TT:11}. 
 Development of tensor techniques for numerical solution of multidimensional problems 
in scientific computing
was promoted  by the reduced HOSVD (RHOSVD) method introduced in \cite{khor-ml-2009}.
It allows to reduce the tensor rank in the canonical format by the canonical-to-Tucker 
(C2T) decomposition without the need to construct the full size tensor. The RHOSVD was initially  
used for the rank reduction in canonical tensors 
in calculation of three-dimensional convolution integrals in computational quantum chemistry, 
see \cite{Khor-book-2018,Khor_bookQC_2018} and the references therein.

In this paper, we  outline the beneficial features  of the tensor-based numerical 
modeling of long-range 
interactions in many-particle systems, which allows to avoid the exponential  complexity  
scaling in dimension, that is typical for the traditional numerical approaches. 
In the recent decade there  was  a burst of new results on  rank-structured
tensor approximation of radial-type multivariate functions arising in computer  simulations of
many-particle  ensembles  \cite{KhKh_CPC:13,VeBokh_NLAA:16,BKK_RS:17,LiKeKKMa:19}.
 In what follows, we generalize the tensor-based approach introduced in 
\cite{KhKh_CPC:13,VeBokh_NLAA:16} which  allows to represent the collective electrostatic 
potential  of a complex lattice-type system on an $n\times n\times n $ 3D Cartesian grid 
in a parametric canonical or Tucker-type  formats  by using a small number of terms
 leading to  $O(n)$ complexity. 
 We discuss how electrostatics in bio-molecular complexes can be  modeled by using the 
RS tensor format, which maintains the long-range part of 
the 3D  charge distribution and the collective potential  in a low-parametric
rank-structured form in $O(n)$-complexity \cite{BKK_RS:17,BKhor_Dirac:18,BeKhKhKwSt:18}. 
 The RS tensor decomposition can be applied to wide classes of radial functions 
providing the low-rank representation of long-range part in large sums of the corresponding
generating kernels.
In this paper we present the RS decomposition for the Slater function and provide the 
comparative analysis with the RS splitting  of the Newton kernel. 
 We show how the RS tensor representation of the free-space potential and electric field allows to 
calculate energy and forces for many-particle systems in low cost. 

The rest of the paper is structured as follows. Section \ref{sec:canon_Newton} presents
 a short overview of classical approaches for numerical modeling of the long-range 
 interaction potentials.
Section \ref{ssec:Coulomb_radial} recalls the construction of the canonical tensor representation 
of the Newton kernel, which was the base for further  developments. 
Section \ref{ssec:direct} discusses the direct tensor summation of Newton kernels in the nuclear
potentials, which  yet does not maintain the full power of tensor-based modeling.
Section \ref{ssec:Tensor_sum_lattice} presents the new results on assembled tensor summation of the 
 interaction  potentials of charged particles centered at nodes of a 3D lattice, 
yielding the linear complexity scaling  for energy calculation  in the univariate lattice size. 
Section \ref{sec:Gener_Mpaticle_pot} describes the main ideas and
advantages of the recent RS tensor format in numerical modeling 
of electrostatics in many-body systems of general type.
 Section \ref{sec:rs-format} discusses the application of RS tensor format in bio-molecular 
 modeling including the energy and force calculations. 
 Section \ref{sec:Append} sketches the main rank-structured tensor formats.

 \section{Canonical tensor approximation of interaction potentials}
 \label{sec:canon_Newton}

 \subsection{Low-rank canonical representation of radial functions}
 \label{ssec:Coulomb_radial}

First, we recall the grid-based method for the low-rank canonical  
representation of a spherically symmetric kernel function $p(\|x\|)$, 
$x\in \mathbb{R}^d$ for $d=2,3,\ldots$, by its projection onto the finite set
of basis functions defined on tensor grid. 
The approximation theory by a sum of Gaussians for the class of analytic potentials 
$p(\|x\|)$  was presented in \cite{Stenger:93,HaKhtens:04I,GHK:05,Khor1:06,Khor-book-2018}.
The particular numerical schemes for rank-structured representation of the Newton 
and Yukawa Green's kernels 
\begin{equation}\label{eqn:NewtYukaw}
p(\|x\|)=\frac{1}{4 \pi \|x\|},\quad  \mbox{  and  }\quad
p(\|x\|)=\frac{e^{-\lambda \|x\|}}{4 \pi \|x\|},\quad
 x\in \mathbb{R}^3,
\end{equation}
discretized on a fine 3D  Cartesian grid
in the form of low-rank canonical tensor was described in \cite{HaKhtens:04I,Khor1:06,BeHaKh:08}.

In what follows, for the ease of exposition, we confine ourselves to the case $d=3$.
In the computational domain  $\Omega=[-b,b]^3$, 
let us introduce the uniform $n \times n \times n$ rectangular Cartesian grid $\Omega_{n}$
with mesh size $h=2b/n$ ($n$ even).
Let $\{ \psi_\textbf{i} =\prod_{\ell=1}^3 \psi_{i_\ell}^{(\ell)}(x_\ell) \}$ 
be a set of tensor-product piecewise constant basis functions, 
labeled by  the $3$-tuple index ${\bf i}=(i_1,i_2,i_3)$, 
$i_\ell \in I_\ell=\{1,...,n\}$, $\ell=1,\, 2,\, 3 $.
The generating kernel $p(\|x\|)$ is discretized by its projection onto the basis 
set $\{ \psi_\textbf{i}\}$
in the form of a third order tensor of size $n\times n \times n$, defined entry-wise as
\begin{equation}  \label{eqn:galten}
\mathbf{P}:=[p_{\bf i}] \in \mathbb{R}^{n\times n \times n},  \quad
 p_{\bf i} = 
\int_{\mathbb{R}^3} \psi_{\bf i} ({x}) p(\|{x}\|) \,\, \mathrm{d}{x}.
\end{equation}

The low-rank canonical decomposition of the $3$rd order tensor $\mathbf{P}$ is based 
on using exponentially convergent 
$\operatorname*{sinc}$-quadratures for approximating the Laplace-Gauss transform 
to the analytic function $p(z)$, $z \in \mathbb{C}$, specified by a certain weight $\widehat{p}(t) >0$,
\begin{align} \label{eqn:laplace} 
p(z)=\int_{\mathbb{R}_+} \widehat{p}(t) e^{- t^2 z^2} \,\mathrm{d}t \approx
\sum_{k=-M}^{M} p_k e^{- t_k^2 z^2} \quad \mbox{for} \quad |z| > 0,\quad z \in \mathbb{R},
\end{align} 
with the proper choice of the quadrature points $t_k$ and weights $p_k$.
 The $sinc$-quadrature based approximation to generating function by 
using the short-term Gaussian sums in (\ref{eqn:laplace})  
are applicable to the class of analytic functions
in certain strip $|z|\leq D $ in the complex plane, such that on the real axis these functions decay
polynomially or exponentially. We refer to basic results in 
\cite{Stenger:93,Braess:BookApTh,HaKhtens:04I}, 
where the exponential convergence of the $sinc$-approximation in the number of terms 
(i.e., the canonical rank) was analyzed for certain classes of analytic integrands.

Now, for any fixed $x=(x_1,x_2,x_3)\in \mathbb{R}^3$, 
such that $\|{x}\| > a > 0$, 
we apply the $\operatorname*{sinc}$-quadrature approximation (\ref{eqn:laplace})  
to obtain the separable expansion
\begin{equation} \label{eqn:sinc_Newt}
 p({\|{x}\|}) =   \int_{\mathbb{R}_+} \widehat{p}(t)
e^{- t^2\|{x}\|^2} \,\mathrm{d}t  \approx 
\sum_{k=-M}^{M} p_k e^{- t_k^2\|{x}\|^2}= 
\sum_{k=-M}^{M} p_k  \prod_{\ell=1}^3 e^{-t_k^2 x_\ell^2},
\end{equation}
providing an exponential convergence rate in $M$,
\begin{equation} \label{eqn:sinc_conv}
\left|p({\|{x}\|}) - \sum_{k=-M}^{M} p_k e^{- t_k^2\|{x}\|^2} \right|  
\le \frac{C}{a}\, \displaystyle{e}^{-\beta \sqrt{M}},  
\quad \text{with some} \ C,\beta >0.
\end{equation}

Combining \eqref{eqn:galten} and \eqref{eqn:sinc_Newt}, and taking into account the 
separability of the Gaussian basis functions, we arrive at the low-rank 
approximation to each entry of the tensor $\mathbf{P}=[p_{\bf i}]$,
\begin{equation*} \label{eqn:C_nD_0}
 p_{\bf i} \approx \sum_{k=-M}^{M} p_k   \int_{\mathbb{R}^3}
 \psi_{\bf i}({x}) e^{- t_k^2\|{x}\|^2} \mathrm{d}{x}
=  \sum_{k=-M}^{M} p_k  \prod_{\ell=1}^{3}  \int_{\mathbb{R}}
\psi^{(\ell)}_{i_\ell}(x_\ell) e^{- t_k^2 x^2_\ell } \mathrm{d} x_\ell.
\end{equation*}
Define the vector (recall that $p_k >0$) 
\begin{equation} \label{eqn:galten_int}
\textbf{p}^{(\ell)}_k
= p_k^{1/3} \left[b^{(\ell)}_{i_\ell}(t_k)\right]_{i_\ell=1}^{n_\ell} \in \mathbb{R}^{n_\ell}
\quad \text{with } \quad b^{(\ell)}_{i_\ell}(t_k)= 
\int_{\mathbb{R}} \psi^{(\ell)}_{i_\ell}(x_\ell) e^{- t_k^2 x^2_\ell } \mathrm{d}x_\ell,
\end{equation}
then the $3$rd order tensor $\mathbf{P}$ can be approximated by 
the $R$-term ($R=2M+1$) canonical representation
\begin{equation} \label{eqn:sinc_general}
    \mathbf{P} \approx  \mathbf{P}_R =
\sum_{k=-M}^{M} p_k \bigotimes_{\ell=1}^{3}  {\bf b}^{(\ell)}(t_k)
= \sum\limits_{k=-M}^{M} {\bf p}^{(1)}_k \otimes {\bf p}^{(2)}_k \otimes {\bf p}^{(3)}_k
\in \mathbb{R}^{n\times n \times n}, \quad {\bf p}^{(\ell)}_k \in \mathbb{R}^n.
\end{equation}
Given a threshold $\varepsilon >0 $, in view of (\ref{eqn:sinc_conv}), we can chose  
$M=O(\log^2\varepsilon )$ such that in the max-norm
\begin{equation*} \label{eqn:error_control}
\| \mathbf{P} - \mathbf{P}_R \|  \le \varepsilon \| \mathbf{P}\|.
\end{equation*}

In the case of Newton kernel we have $p(z)=1/z$, $\widehat{p}(t)=\frac{2}{\sqrt{\pi}}$, so that
the Laplace-Gauss transform representation reads
\begin{equation} \label{eqn:Laplace_Newton}
 \frac{1}{z}= \frac{2}{\sqrt{\pi}}\int_{\mathbb{R}_+} e^{- z^2 t^2 } dt, \quad 
 \mbox{where}\quad z=\|x\|, \quad x \in \mathbb{R}^3,
\end{equation}
which can be approximated by the sinc quadrature (\ref{eqn:sinc_Newt}) with the particular choice 
of quadrature points $t_k$, providing the exponential convergence rate as 
in (\ref{eqn:sinc_conv}), \cite{HaKhtens:04I,Khor1:06}.

In the case of Yukawa potential the Laplace Gauss transform reads   
 \begin{equation} \label{eqn:Laplace_Yukawa}
  \frac{e^{-\kappa z}}{z}= \frac{2}{\sqrt{\pi}}\int_{\mathbb{R}_+} e^{-\kappa^2/t^2} e^{- z^2 t^2 } dt, \quad 
 \mbox{where}\quad z=\|x\|, \quad x \in \mathbb{R}^3.
 \end{equation}
 The analysis of the sinc quadrature approximation error for this case can 
 be found in particular in \cite{Khor1:06,Khor-book-2018}, \S2.4.7.

 One can observe from numerical tests that there are canonical  vectors representing the long- 
and short-range (highly localized) contributions to the total electrostatic potential. 
This interesting feature was also recognized for the rank-structured 
tensors representing a lattice sum of electrostatic potentials 
\cite{KhKh_CPC:13,VeBokh_NLAA:16}. 

\begin{table}[tbh]\label{tab_Newton}
 \begin{center}%
\begin{tabular}
[c]{|r|r|r|r|r|r|}%
\hline
grid size $n^3 $ & $8192^3$  & $16384^3$ & $32768^3$ & $65536^3$ & $131072^3$\\ 
 \hline 
 Time (s)     &  $1 $    &    $2$  &   $8 $   &  $43$    & $198$ \\
 \hline 
Canonical rank $R_N$ &  $34$ &   $36$  &   $38 $   &  $40$    & $42$ \\
 \hline 
Compression rate  & $2\cdot 10^6$ & $7\cdot 10^6$ & $2\cdot 10^7$ & $1\cdot 10^8$ & $4\cdot 10^8$  \\
 \hline 
 \end{tabular}
 \end{center}
\caption{\small CPU time for generation of the canonical tensor representation 
of the Newton kernel ${\bf P}$, with accuracy $\varepsilon = 10^{-6}$. }
\end{table} 
Table \ref{tab_Newton} demonstrates the times (using Matlab) for generating the 
Newton kernel in computational box with a given size of the 3D Cartesian grid. 
The accuracy (tolerance error) is chosen as $\varepsilon = 10^{-6}$. 
Table \ref{tab_Newton} shows that the low-parametric tensor representation requires  
feasible computation time for large  sizes of the three-dimensional grid and provides
huge compression rate compared with the full size tensor.

Initially the representation of the convolving kernel $\frac{1}{\|{ x} \|}, \; x  \in \mathbb{R}^3$ 
by a low-rank canonical tensor  was used in calculation of the 3D convolution integral 
operators in 1D complexity   
in electronic structure calculations 
\cite{khor-ml-2009,KhKhFl_Hart:09,Khor_bookQC_2018}.

 \subsection{Many particle potential via direct tensor summation}
 \label{ssec:direct}
 
 Calculation of the nuclear potential operator in the tensor-based Hartree-Fock solver
 is performed by using the canonical tensor representation of the 
 Newton kernel \cite{Khor_bookQC_2018}.
 Nuclear potential operator in the Hartree-Fock equation is computed as the   
 electrostatic potential created by all nuclei in a molecule,  
 \begin{equation}\label{eqn:Electrost_sum}
V_c(x)= - \sum_{a=1}^{N}\frac{Z_a}{\|{x} -x_a \|},\quad
Z_a \in \mathbb{R}, \;\; x, x_a\in \mathbb{R}^3,
\end{equation}
where $N$ is the number of nuclei and $Z_a$ are their charges.

When using the tensor-based numerical approach, the nuclear potential operator 
is calculated in a computational box
$[-b/2,b/2]^3$, using the $n \times n \times n $ 3D Cartesian grid.
First, the reference canonical rank-$R$ tensor of size $2n \times 2n \times 2n $ representing 
the Newton kernel is generated in a box of a double size $[-b,b]^3$,
\begin{equation}
\label{eqn:NewtCan}
\widehat{\bf P}_{R}= 
\sum\limits_{q=1}^{R} \widehat{\bf p}^{(1)}_q \otimes \widehat{\bf p}^{(2)}_q 
\otimes \widehat{\bf p}^{(3)}_q.
\end{equation}
In the course of potential calculation
the reference tensor $\widehat{\bf P}_R$ should be translated to all nuclei positions
by the shifting (windowing) operator 
\begin{equation}
\label{Sh-Win}
 {\cal W}_{a}={\cal W}_{a}^{(1)} \otimes {\cal W}_{a}^{(2)}\otimes {\cal W}_{a}^{(3)},
\end{equation}
which is constructed for every nucleus $a$ in a molecule, to shift the reference Newton kernel 
to the number of grid-points 
corresponding to coordinates of a nucleus in variables $x$, $y$ and $z$, and then to
make a cut-off of the shifted tensor to the computational box of size $[-b/2,b/2]^3$, see details 
in  \cite{Khor_bookQC_2018}. The upper indexes $(1)$, $(2)$
and $(3)$ of the operators ${\cal W}_{a}$ in (\ref{Sh-Win}) 
correspond to the axes  $x$, $y$ and $z$, respectively.
  
The resulting electrostatic potential is composed as a sum of canonical tensors 
which are  
obtained by application of the corresponding scaled (with the charges of 
the corresponding nuclei) shifting-windowing operators 
to the reference Newton kernel,
\begin{equation}
\label{eqn:Direct_sum}
 V_c \approx{\bf P}_{c}  = \sum_{a=1}^{N} Z_a {\cal W}_{a} \widehat{\bf P}_R 
              =\sum_{a=1}^{N} Z_a  
\sum\limits_{q=1}^{R} {\cal W}_{a}^{(1)} \widehat{\bf p}^{(1)}_q \otimes 
{\cal W}_{a}^{(2)} \widehat{\bf p}^{(2)}_q 
\otimes {\cal W}_{a}^{(3)} \widehat{\bf p}^{(3)}_q\in \mathbb{R}^{n\times n \times n}.
\end{equation}
In general,
 the rank of the resulting canonical tensor ${\bf P}_c$ is proportional to the number 
 of nuclei, $N$, i.e.,   max(Rank(${\bf P}_{c}$)) = $R N$.
    It can be reduced by using the combination of the canonical-to-Tucker and the 
  Tucker-to-canonical transforms. The accuracy of this rank reduction depends on the chosen
  $\varepsilon$-threshold for the tensor transforms. Hence, the direct tensor summation 
  of the potentials maintaining the high accuracy leads, 
  in general, to unacceptably large ranks of the resulting tensor.
  
  In what follows we explain how to get rid of this drawback in the framework of 
  rank-structured tensor calculations in case of both lattice-structured and unstructured 
  locations of particles.

 \section{Assembled tensor summation for particles on a lattice }
 \label{sec:lattices}
 
 \subsection{Sketch of classical approaches}
  \label{ssec:classic}
  
  Let us consider the well-known expressions for the electrostatic potentail of a number of charges
  particles and their interaction energy. The collective electrostatic potentials generated by $N$
  charged particles  is defined by (\ref{eqn:Electrost_sum})),  
while the energy of their electrostatic interaction is given by
  \begin{equation}
 \label{inter_energy1} 
 E_{nuc}= \sum^{N}_{i=1} \sum^{N}_{j <i} \frac{Z_i Z_j}{\|x_i- x_j\|}.  
 \end{equation}
 When using the traditional numerical approaches, the potential (\ref{eqn:Electrost_sum}))
  may be computed for every point $x\in\mathbb{R}^3$ on $n\times n \times n$ grid separately, 
  yielding the extensive computational work of the order of $O(N n^3)$.
  On the other hand, the low-rank parametric representation of the collective 
  potential of $N$ generally distributed charged particles discretized over $n\times n \times n$ 
  grid   is not tractable. 
  The straightforward calculation of the energy by (\ref{inter_energy1}) amounts to 
  $O(N^2)$ operations.
  
 If particles are located on an $L \times L\times L$ 
 lattice (i.e., $N=L^3$) the energy calculation an be  accelerated.  
 The traditional method of Ewald-type
 summation \cite{Ewald:27} is based on a specific local-global decomposition  of the Newton kernel,
\[
 \frac{1}{r}=\frac{\tau(r)}{r} + \frac{1-\tau(r)}{r}, \quad r=\|x\|,
\]
  where the cutoff function $\tau$ is chosen as the complementary error function
\[ 
\tau(r)=\operatorname{erfc}(r):=\frac{2}{\sqrt{\pi}}\int_{r}^\infty \exp(-t^2) dt.   
 \]
In the Ewald summation method \cite{Ewald:27,DYP:93,PolGlo:96,Hune_Ewald:99,HuMcCam:1999} the computation 
of the interaction energy on the long-range term is performed in the 
Fourier space (using periodic boundary conditions). It allows to reduce computational 
work for calculation of the interaction energy of particles placed in the nodes of the lattice
from $O(L^6)$ to $O(L^3 \log L)$. 

%

The fast multipole method  (FMM) \cite{RochGreen:87} is used for computation of the interaction energy of 
charged multiparticle systems of general type at the expense  $O(N \log^q N)$.
Notice that, both Ewald summation method and FMM operate only with the values of interaction 
potential at the particle centers, but they 
are not capable to calculate and store the potential function on the fine spacial 
$n \times n \times n$ grid at the cost that is noticeably lower than $O(N^2 n^3)$.

%
 
In the following section, we describe the novel tensor based techniques for fast 
grid-based calculation of electrostatic 
potential of a lattice-structured system and for efficient recovery of the many-particle
interaction energy.
 
\subsection{Low rank potential sum over rectangular lattice}
  \label{ssec:Tensor_sum_lattice}

Recently the rank-structured tensor approach suggested a progress in the numerical treatment of the 
long-range electrostatic potentials  in many-particle systems. 
Here, we recall the 
method for fast summation of the electrostatic potentials placed on large 3D lattice,
introduced by the authors in  \cite{KhKh_CPC:13,Khor_bookQC_2018},
and generalize it to the case of rather general distribution of charges 
in $d$ -dimensional setting.
We will show, that given the canonical tensor representation of the generating kernel, the 
collective many-particle potential on a $d$-dimensional rectangular lattice can be
computed at the cost that is almost linear (or quadratic) proportional to the univariate lattice size
independently on the number of dimensions $d$.

Let us consider first a sum of single Coulomb potentials with equal point charge $Z$ on a 3D
finite $L \times L\times L$ lattice in a volume box $\Omega_0=[-b/2,b/2]^3$,
\begin{equation}\label{eqn:CoulombLattice}
V_{c_L}(x)=  \sum_{k_1,k_2,k_3\in {\cal K}} 
\frac{Z}{\|{x} -a_1 (k_1,k_2,k_3)\|}, \quad x\in \Omega_L\in \mathbb{R}^3,
\end{equation}
where ${\cal K}=\{k: \, 1\leq k \leq L \}$.
The assembled tensor summation method applies to the lattice potentials 
defined on a fine $n\times n\times n$ 3D Cartesian grid that embeds also the lattice points. 
This techniques reduces the calculation of the collective potential sum over a rectangular 
3D lattice, 
\[
  {\bf P}_{c_L}= \sum\limits_{{\bf k} \in {\cal K}^{\otimes 3}}  { W}_{\nu({\bf k})} \widehat{\bf P}
= \sum\limits_{k_1,k_2,k_3 \in {\cal K}}  {W}_{ ({\bf k})} \sum\limits_{q=1}^{R}  
( \widehat{\bf p}^{(1)}_{q} \otimes \widehat{\bf p}^{(2)}_{q} 
\otimes \widehat{\bf p}^{(3)}_{q}) \in \mathbb{R}^{n\times n  \times n},
\]
\begin{figure}[tbh]
\centering  
\includegraphics[width=6cm]{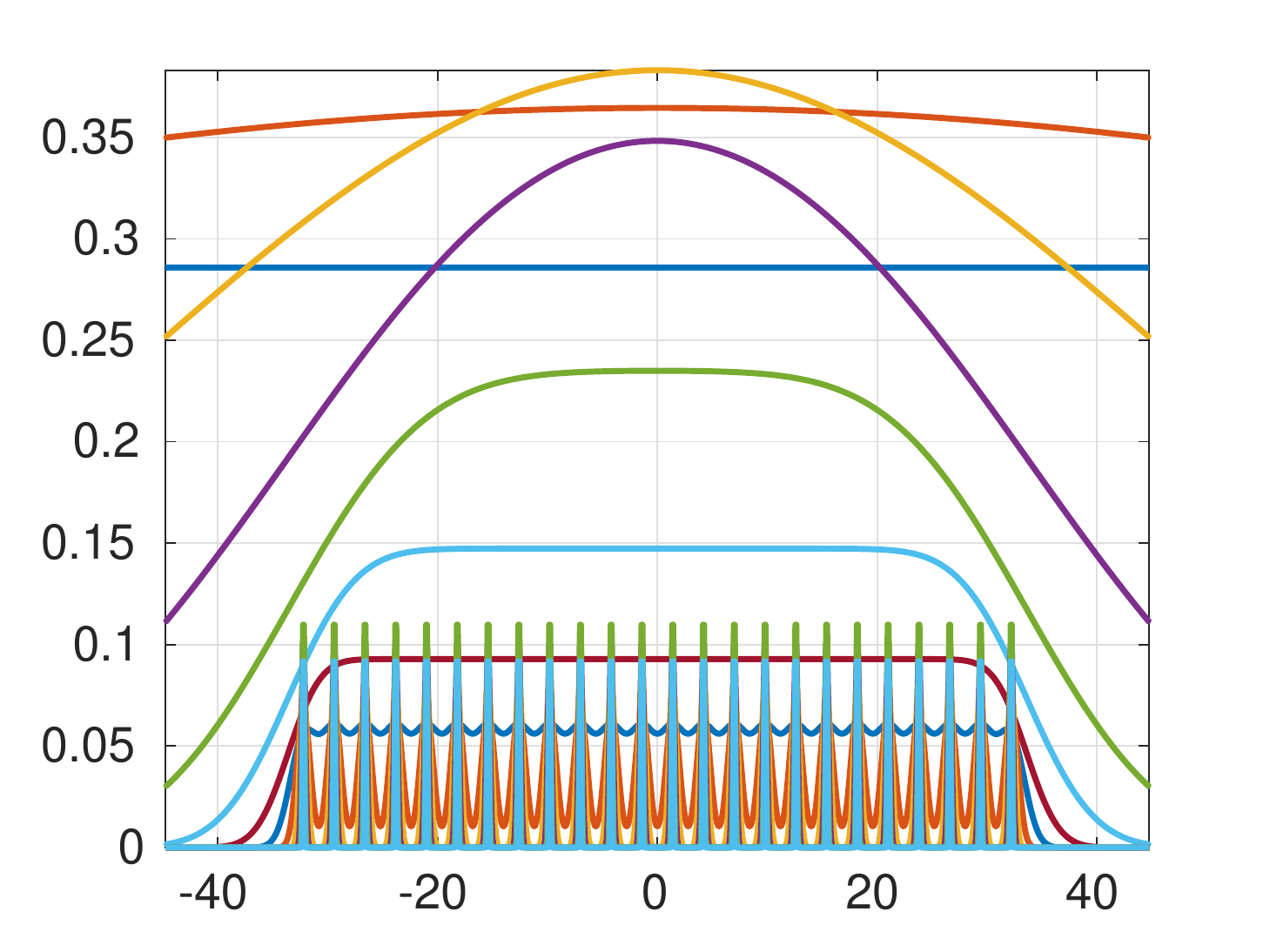}\quad\quad\quad
\includegraphics[width=6cm]{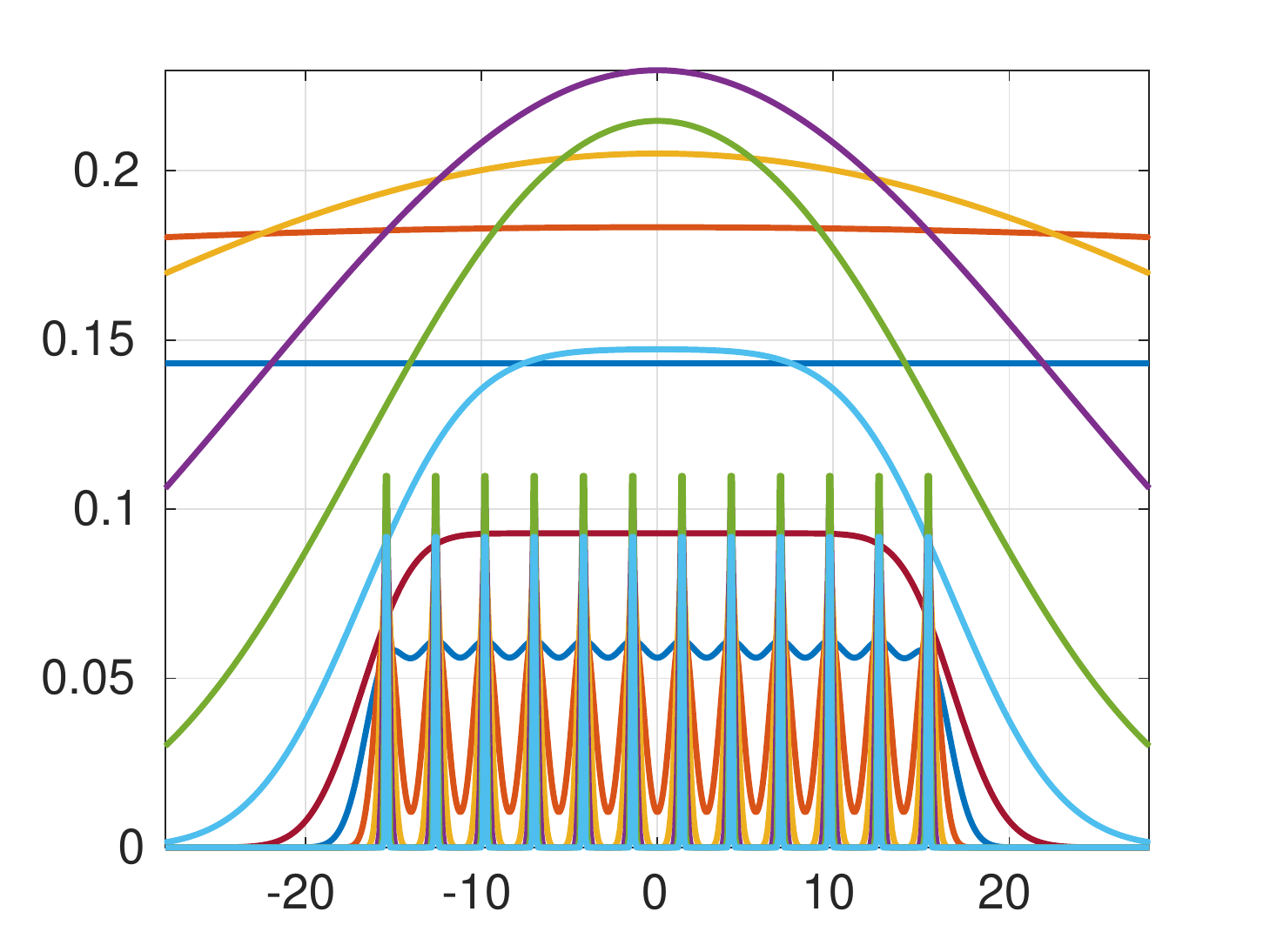}\\
\includegraphics[width=6cm]{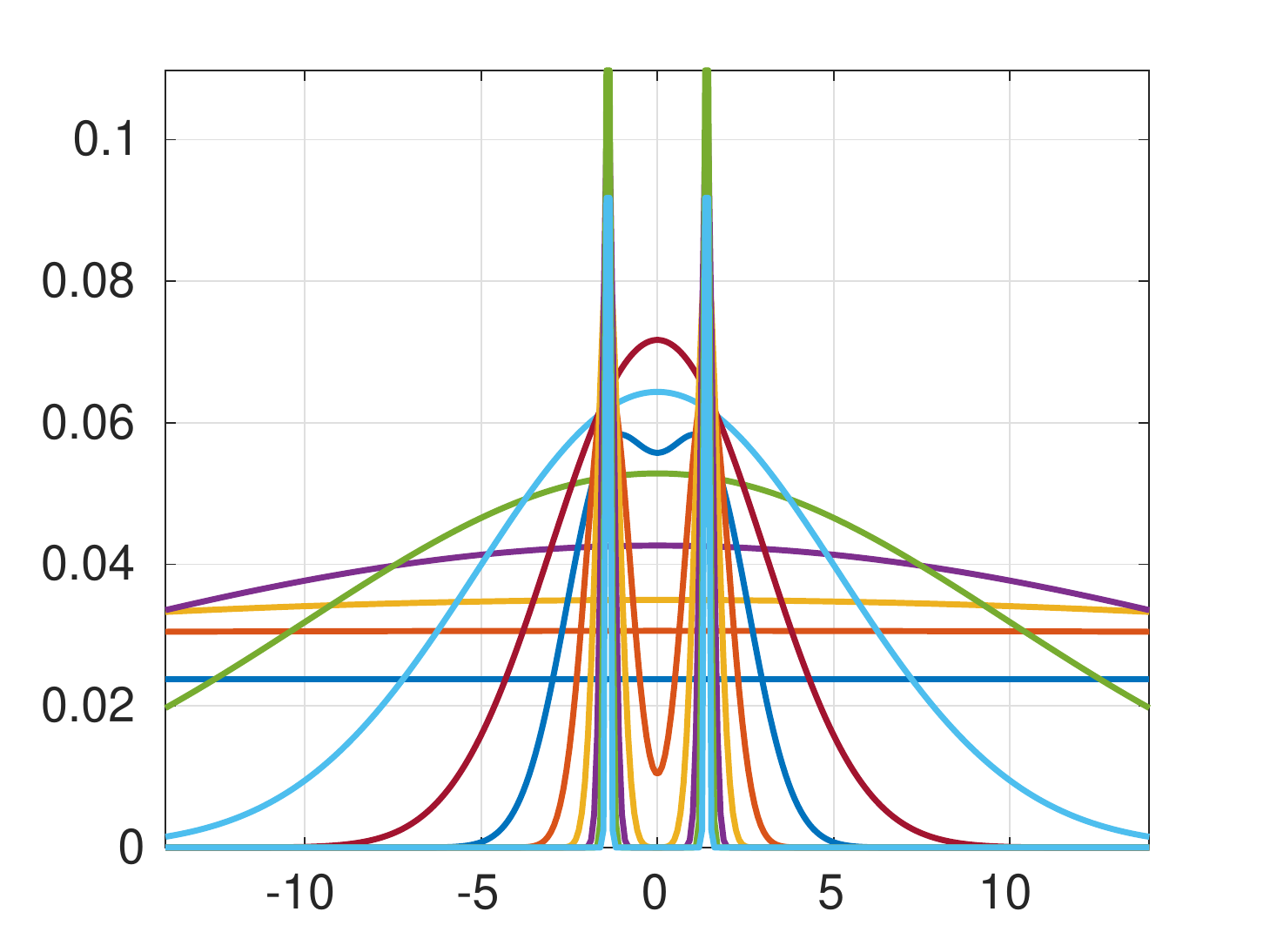}\quad
\includegraphics[width=7cm]{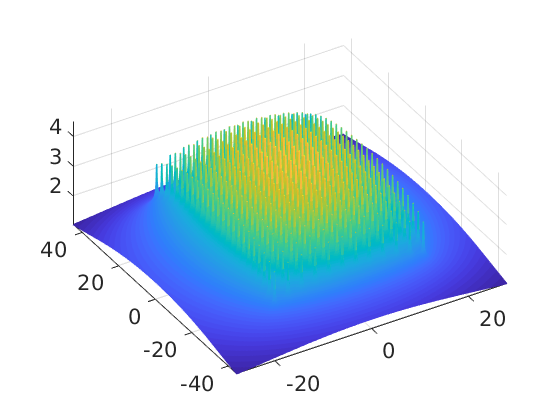}
\caption{Assembled $x$-, $y$- and $z$-axis canonical vectors for a cluster of 
$24\times 12\times 2$ Hydrogen atoms and  the collective 3D electrostatic potential 
at the cross-section corresponding to one of z-axis layers.} 
\label{fig:Assembl_can}  
\end{figure}
to the summation (assembling) of shifted directional vectors of the canonical tensor 
representation for a single Newton kernel \cite{KhKh_CPC:13}, 
\begin{equation}\label{eqn:CoulombLatTensCan}
{\bf P}_{c_L}=  
\sum\limits_{q=1}^{R}
(\sum\limits_{k_1\in {\cal K}} { W}_{({k_1})}  \widehat{\bf p}^{(1)}_{q}) \otimes 
(\sum\limits_{k_2\in {\cal K}} { W}_{({k_2})} \widehat{\bf p}^{(2)}_{q}) \otimes 
(\sum\limits_{k_3\in {\cal K}}{ W}_{({k_3})}  \widehat{\bf p}^{(3)}_{q}).
\end{equation}
Here $\widehat{\bf P}$ (similar to (\ref{eqn:NewtCan})) represents 
the single Newton kernel on a twice larger grid and 
\begin{equation}\label{wind_latt}
 { W}_{\nu({\bf k})}= { W}_{(k_1)}\otimes { W}_{(k_2)}\otimes { W}_{(k_3)}
\end{equation}
is the shift-and-windowing (onto $\Omega_L$) separable transform along the ${\bf k}$-grid,
analogous to the operator (\ref{Sh-Win}) for direct tensor summation.
For rectangular finite 3D lattices the rank of the resulting sum is proven 
to be the same as for the $R$-term 
canonical tensor $\widehat{\bf P}$ representing a single Coulomb potential \cite{KhKh_CPC:13}.

Figure \ref{fig:Assembl_can} shows the shapes of the assembled canonical vectors for the 
tensor-based summation of the Hydrogen-nuclei potentials on a 
rectangular two-layers lattice with of size $24 \times 24 \times 2$. 
Left bottom figure shows the electrostatic potential at the
cross-section perpendicular to the $z$-axis, which corresponds to the upper layer of the charges. 

In the following statement, we generalize this theorem to the case of 
$d$-dimensional rectangular lattices, thus avoiding the curse of dimensionality.
\begin{theorem}\label{thm:d_lattice}
Given rank-$R$ reference canonical tensor ${\bf P}$, representing a single reference potential, 
the collective interaction potential $v_{c_L} (x)$, $x\in \Omega_{L}$,  
 generated by point potentials on $L^{\otimes d}$ rectangular lattice  is presented by the 
canonical tensor ${\bf P}_{c_L}$ of the same rank $R$,
\begin{equation}\label{eqn:EwaldTensorGl}
{\bf P}_{c_L}= 
Z \sum\limits_{q=1}^{R}
(\sum\limits_{k_1=1}^{L}{\cal W}_{({k_1})} {\bf p}^{(1)}_{q}) \otimes 
(\sum\limits_{k_2=1}^{L} {\cal W}_{({k_2})} {\bf p}^{(2)}_{q}) \otimes \cdots \otimes 
(\sum\limits_{k_d=1}^{L}{\cal W}_{({k_d})} {\bf p}^{(d)}_{q}),
\end{equation}
where ${\cal W}_{({k_\ell})}$, $\ell=1,\,2,\,\ldots,\, d$ are the shifting-windowing operators for
the $d$-dimensional lattice $L\times L\times \cdots \times L$.
The numerical cost for evaluation of this potential on $n^{\otimes d}$ tensor grid is 
estimated by $O(d R L n)$.
\end{theorem} 
The proof is similar to that given in \cite{KhKh_CPC:13}, and it is based on  separability  
of the shifting-windowing operator (\ref{wind_latt}) on rectangular lattices.

For 3D lattices with multiple vacancies, the tensor rank increases by a 
small factor  \cite{VeBokh_NLAA:16}. 
Indeed, for composite geometries the electrostatic potential  over defected lattices with  
simple inclusions can be represented by a sum of low-rank tensors supported on the particular 
vacancies, ${\cal L}_q$, which compose the total lattice ${\cal L}$, 
\[
 {\cal L}= \bigcup {\cal L}_q.
\]
 An example of the collective electrostatic potential over a composite lattice with two impurities
with different interatomic distances is shown in Figure \ref{fig:compos_grid}. 
One of the impurities is defined on the sub-lattice of size $3\times 3\times 1$, 
with a larger positive point charge,
while the second one is supported by a sub-lattice of size $3\times 3\times 1$, and 
contains negative point charges.
Further examples can be found in \cite{Khor_bookQC_2018},
where the assembled tensor summation is discussed in more detail.

In fact, the assembled tensor summation for lattices
admits modeling of a number of such impurities since they are all living on the same
$n\times n\times n $ 3D Cartesian grid. Then such  low-rank canonical tensor representation 
the collective electrostatic potential 
enables operations like 3D convolution on 3D grids with $O(n)$ complexity.
\begin{figure}[tbh] \label{fig:compos_grid}
\centering  
\includegraphics[width=6.0cm]{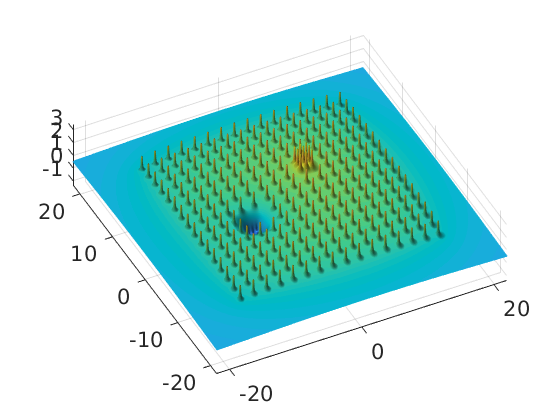}  \quad
 \includegraphics[width=6.0cm]{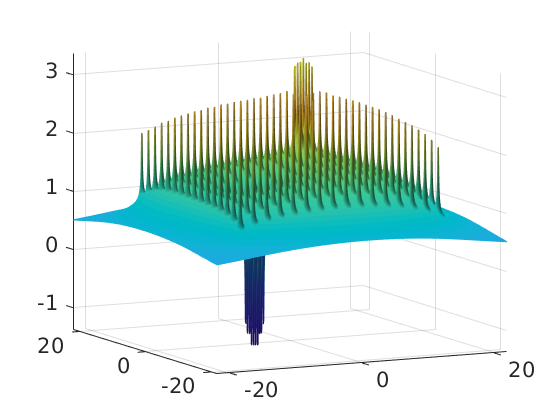}
\caption{\small  A cross-section of the potential at the plane with $z=0$ for the 
lattice of size $16\times 16\times 1$ with two impurities 
having different interatomic distances and different charges.
The left and right figures show  the views from different angles.}
  \end{figure}

 \begin{table}[tbh]\label{Tab:times}
\begin{center}%
\begin{tabular}
[c]{|r|r|r|r|r|}%
\hline
  $L^3$ & $64^3$  & $128^3$ & $256^3$ & $512^3$ \\ 
 \hline 
   number of  particles  & 262144  & 2097152 & $16\cdot 10^6$ & $134\cdot 10^6$ \\ 
   \hline 
  size in nanometers$^3$ &  $7^3$ & $13.4^3$ & $26.2^3 $ & $51.8^3$ \\
   \hline\hline 
     $n^3 $  & $4480^3$ & $8576^3$ & $16768^3$ & $33152^3$ \\ 
   \hline  
  time              &  $1.2$   &   $2.4$  &   $8.9 $   & $37.2 $   \\
\hline\hline
   $n^3 $  & $8960^3$ & $17152^3$ & $33536^3$ & $66304^3\quad (10^{14})$  \\ 
 \hline 
time               &  $1.0$    &    $4.1$  &   $18.8 $  & $84.6$   \\
\hline 
 \end{tabular}
  \end{center}
  \caption{\small Times (sec.) for calculation of the free-space collective electrostatic potential of
  charged particles   placed in nodes of a $L\times L \times L$  3D lattice 
  vs. the size  of representation grid, $n^3$.}. 
     \end{table}
The numerical implementation of the representation (\ref{eqn:EwaldTensorGl}) 
for the 3D collective electrostatic potential is extremely fast 
due to the fact that summation
%
of the  total potential is reduced to shifting and summation of vectors. 
Table \ref{Tab:times} shows the computation times for summation of charges on large 3D 
lattices composed by nuclei of Hydrogen atoms with  an inter-atomic distance of 1 bohr (atomic unit).
For example a cluster of $512^3$ particles corresponds to a domain size $51.8^{\otimes 3}$ nanometers.
We notice, that these calculations are performed in real space (not in the frequency domain).  

Now we generalize the previous scheme to the case of variable charges.
Consider the 3D case and introduce the three-fold charge tensor ${\bf Z}=\{z_{k_1,k_2,k_3}\}$,
$k_\ell=1,\ldots,L$. Assume that tensor ${\bf Z}$ admits the rank-$R_Z$ canonical decomposition
\begin{equation}\label{eqn:charges_weight}
 {\bf Z}=\sum\limits_{m=1}^{R_Z} {\bf z}^{(1)}_m \otimes {\bf z}^{(2)}_m \otimes {\bf z}^{(3)}_m, 
 \quad {\bf z}^{(\ell)}_m\in \mathbb{R}^L.
\end{equation}
Consider the weighted electrostatic potential
\begin{equation}\label{eqn:CoulombLatt_weight}
V_{c_L}(x)=  \sum_{k_1,k_2,k_3\in {\cal K}} 
\frac{z_{\bf k}}{\|{x} -a_1 (k_1,k_2,k_3)\|}, \quad x\in \Omega_L\in \mathbb{R}^3,
\end{equation}
and the corresponding grid-based discretization  represented as the canonical tensor 
\[
  {\bf P}_{c_L}= \sum\limits_{{\bf k} \in {\cal K}^{\otimes 3}} 
  z_{\bf k} { W}_{\nu({\bf k})} \widehat{\bf P}
= \sum\limits_{k_1,k_2,k_3 \in {\cal K}} z_{\bf k} {W}_{ ({\bf k})} 
\sum\limits_{q=1}^{R}   ( \widehat{\bf p}^{(1)}_{q} \otimes \widehat{\bf p}^{(2)}_{q} 
\otimes \widehat{\bf p}^{(3)}_{q}) \in \mathbb{R}^{n\times n  \times n}.
\]

We prove that tensor ${\bf P}_{c_L}$ 
can be calculated in the low-rank canonical format and stored in almost linear cost in $L$, 
provided that (\ref{eqn:charges_weight}) holds.
\begin{theorem}\label{thm:d_lattice_Variable}
Given the rank-$R$ reference canonical tensor ${\bf P}$ representing the single Newton kernel, 
the collective interaction potential ${\bf P}_{c_L}$  
generated by a sum of weighted potentials over $L^{\otimes 3}$ rectangular lattice,  
where weights ${\bf Z}$ are given in the low-rank form (\ref{eqn:charges_weight}),
can be presented by the canonical rank $\leq R_Z R$ tensor
\begin{equation}\label{eqn:EwaldTensorGl_Coeff}
{\bf P}_{c_L}= 
  \sum\limits_{m=1}^{R_Z}    \sum\limits_{q=1}^{R}
(\sum\limits_{k_1=1}^{L} z^{(1)}_{m} {\cal W}_{({k_1})} {\bf p}^{(1)}_{q}) \otimes 
(\sum\limits_{k_2=1}^{L} z^{(2)}_{m}{\cal W}_{({k_2})} {\bf p}^{(2)}_{q}) \otimes  
(\sum\limits_{k_d=1}^{L} z^{(3)}_{m} {\cal W}_{({k_3})} {\bf p}^{(3)}_{q}).
\end{equation}
The numerical cost for evaluation of the potential (\ref{eqn:EwaldTensorGl_Coeff}) 
on $n^{\otimes 3}$ tensor grid is estimated by $O( R_Z R L n)$.
\end{theorem}
\begin{proof}
 Recall that the separable 3D shift operator is defined by $
{\cal W}_{({\bf k})}={\cal W}_{ (k_1)}^{(1)}\otimes {\cal W}_{ (k_2)}^{(2)}
\otimes {\cal W}_{ (k_3)}^{(3)}$,   ${\bf k}\in \mathbb{Z}^{L\times L\times L}$.
Now the representation (\ref{eqn:charges_weight}) implies
\[
\begin{split}
 {\bf P}_{c_L} &= \sum\limits_{q=1}^{R} 
 \sum\limits_{k_1,k_2,k_3=1}^{L}  z_{\bf k}
   {\cal W}_{ (k_1)}^{(1)} {\bf p}^{(1)}_{q} \otimes {\cal W}_{ (k_2)}^{(2)}{\bf p}^{(2)}_{q} 
 \otimes {\cal W}_{ (k_3)}^{(3)} {\bf p}^{(3)}_{q}  \\
 & = 
  \sum\limits_{q=1}^{R}  \sum\limits_{m=1}^{R_Z}  \left[
(\sum\limits_{k_1=1}^L z^{(1)}_{m} {\cal W}_{ ({k_1})} {\bf p}^{(1)}_{q}) \otimes
 (\sum\limits_{k_2,k_3=1}^L z^{(2)}_{m} {\cal W}_{ ({k_2})} {\bf p}^{(2)}_{q} 
\otimes z^{(3)}_{m} {\cal W}_{ ({k_3})} {\bf p}^{(3)}_{q})\right]\\
 &= \sum\limits_{q=1}^{R} \sum\limits_{m=1}^{R_Z}
 (\sum\limits_{k_1=1}^L z^{(1)}_{m} {\cal W}_{ ({k_1})} {\bf p}^{(1)}_{q}) \otimes 
(\sum\limits_{k_2=1}^L  z^{(2)}_{m} {\cal W}_{ ({k_2})} {\bf p}^{(2)}_{q}) \otimes 
(\sum\limits_{k_3=1}^L  z^{(3)}_{m}{\cal W}_{ ({k_3})} {\bf p}^{(3)}_{q}).
\end{split}
\]
This proves the theorem. 
\end{proof}

This theorem can be easily generalized to the $d$-dimensional case, cf. Theorem \ref{thm:d_lattice}.

%


We conclude that the presented method of grid-based assembled tensor summation of 
the weighted electrostatics 
 potentials represented on $L^{\otimes d} $ lattice can be performed at 
 the numerical cost and storage size bounded by 
$O(d R_Z R L n )$ and $O(d R_Z R n)$, respectively.
 
 \begin{figure}[tbh]
\centering  
\includegraphics[width=7cm]{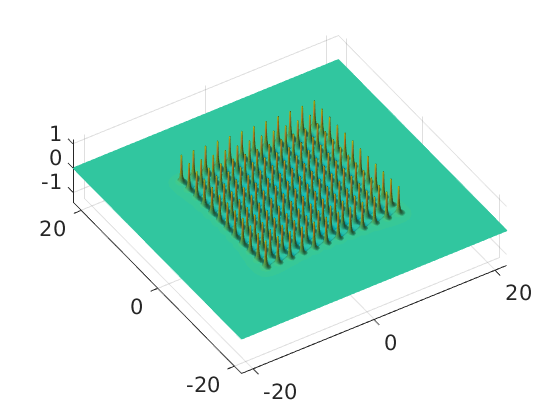}\quad 
\includegraphics[width=7cm]{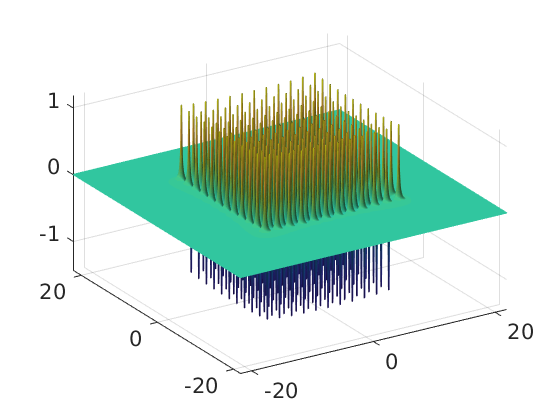}\\
\includegraphics[width=6cm]{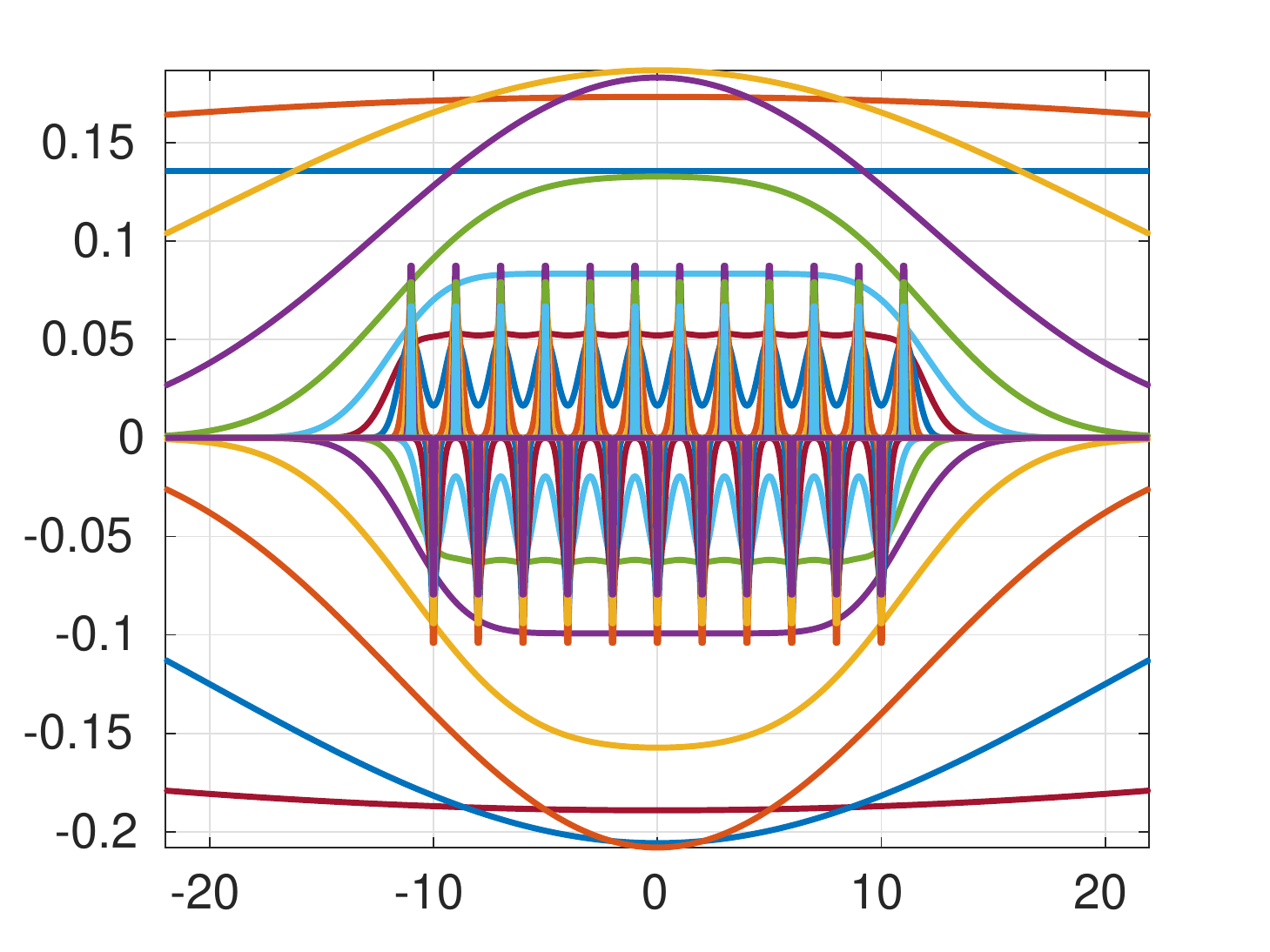}\quad
\includegraphics[width=6cm]{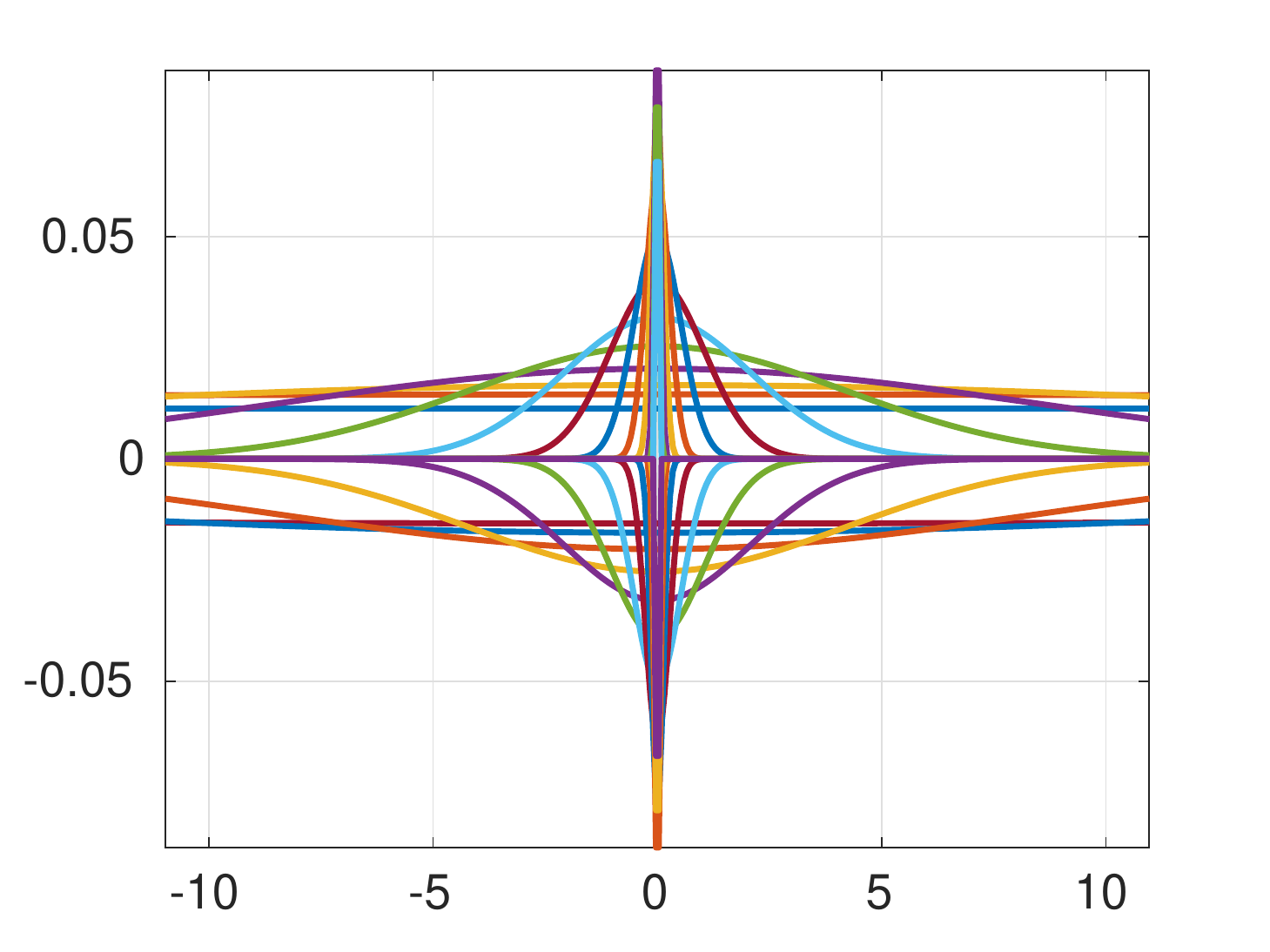}
\caption{\small Collective 3D electrostatic potential  modeling one-layer dipole-type   system 
including 
$12\times 12 \times 1$ positive and $11\times 11\times 1$ negative charges. Bottom figures show the
$x$- and $z$-axis canonical vectors for this cluster.} 
\label{fig:Assembl_dip}  
\end{figure}
Figure \ref{fig:Assembl_dip} presents the electrostatic potential of a dipole-type
lattice with positive and negative charges mixed in a checkerboard order, see 
top figures. In this example the charge configuration is represented by rank-$2$
tensor array. Bottom figures show shapes of the assembled canonical vectors 
resulted by the tensor-based summation of the potentials on a lattice.
One can distinguish separate picks for positive and negative contributions.
The values of the negative charges are chosen to provide  a neutral system.

 \subsection{ Interaction energy of charged particles on 3D lattices }
 \label{ssec:lattice_energy}
 
   Notice that the traditional approaches, like Ewald-type summation \cite{Ewald:27}
 exhibit $O(L^3 \log L)$ complexity for the energy calculations 
 on $L\times L \times L$ 3D lattice.
 In turn calculation of the interaction energy of equally charged particles 
 in 3D case by using tensor representation of the collective potential scales linearly in $L$
 \cite{KhKh_CPC:13}.
 In this section we consider this issue in the more general setting.
 
  The interaction energy of charged particles in a $L\times L \times L$ lattice is given by 
 \[
 E_{nuc}= \frac{1}{2}\sum^{L^3}_{i=1} \sum^{L^3}_{j=1,j\neq i} \frac{z_i z_j}{\|x_i- x_j\|},
 \quad x_i, x_j\in\mathbb{R}^3,
  \]
  where $z_i$ and $z_j$ are the particle charges and $x_i$ and $x_j$ are their coordinates.
  $O(3 R L) \ll L^3 \log L$ 
Recall that the interaction energy of a systems of equal-charged particles in a lattice can be efficiently calculated
 by using the collective electrostatic potential of a lattice-type computed using the assembled
 vectors of the canonical (or Tucker) tensor representation of the Newton kernel.
 
Given $\widetilde{\bf P}_{c_L}$ that is the trace 
of the (rescaled by the factor $ h^3$) collective potential ${\bf P}_{c_L}$
onto the lattice points, 
and denote the all-ones tensor of the same size by ${\bf 1}$,
then the energy sum with accuracy $O(h^2)$ is computed as (see \cite{KhKh_CPC:13})
\[
E_{L,T} = \frac{Z^2}{2} (\langle \widetilde{\bf P}_{c_L}, {\bf 1}\rangle  - 
 \sum\limits_{{\bf k}\in {\cal K}} {\bf P}_{|x_{\bf k}=0} ),
\]
where $Z$ is the point charge.
Finally, by introducing the rank-$1$ tensor ${\bf P}_{0 L} ={\bf P}_{|x_{\bf k}=0}{\bf 1}$,
we represent the interaction energy $E_{L,T}$  by a simple tensor operation
\begin{equation}\label{eqn:EnergyLattice}
 E_{L,T}  = \frac{Z^2 }{2} (\langle \widetilde{\bf P}_{c_L}, {\bf 1}\rangle - 
\langle {\bf P}_{0 L}, {\bf 1}\rangle),
\end{equation}
which can be implemented in $O(3 R L) \ll L^3 \log L$ complexity. Indeed, Theorem \ref{thm:d_lattice}
implies that the tensor $\widetilde{\bf P}_{c_L}$ has the canonical rank not larger than $R$.

Likewise, in the case of variable charges, Theorem \ref{thm:d_lattice_Variable} justifies that
the canonical rank of the corresponding collective potential $\widetilde{\bf P}_{c_L}$ 
does not exceed $R_Z R$. Hence, in this case the representation (\ref{eqn:EnergyLattice}) 
takes the form
\begin{equation}\label{eqn:EnergyLattice_variable}
 E_{L,T}  = \frac{1}{2} (\langle \widetilde{\bf P}_{c_L}, {\bf Z}\rangle - 
 {\bf P}_{|x_{\bf k}=0}\langle {\bf Z},{\bf Z} \rangle),
\end{equation}
which can be evaluated in $O(3 R_Z R L)$ operations.
The above arguments prove the following result.
\begin{theorem}\label{thm:d_lattice_Var_energy}
Given the tensor $\widetilde{\bf P}_{c_L}$, in the case of constant charges 
on the $L\times L \times L$-lattice the interaction energy 
of the lattice-structured system is calculated by (\ref{eqn:EnergyLattice}) in the linear 
cost in $L$, $O(3 R L)$. In the case of variable charge distributions with  
the charges given by a rank-$R_Z$ tensor, the interaction energy of the lattice-structured system is calculated by
(\ref{eqn:EnergyLattice_variable}) at the expense $O(3 R_Z R L)$.
\end{theorem}

Table \ref{Tab:energy} shows CPU times for computation of the interaction energy ($Z=1$) for several
cubic clusters of charged particles using Matlab. Time T$_{full}$ denotes the time for direct $O(L^6)$
computations, while  $T_{(L^3)}$ presents the square root of this time, to show the  $O(L^3)$
complexity (restricted to $L=48$). The column  T$_{tens.}$ shows the times for tensor-based calculations.
\begin{table}[tbh]\label{Tab:energy}
\begin{center}%
\begin{tabular}
[c]{|r|r|r|r|r|}%
\hline
$L^3$ &   T$_{full}$ $T_{(L^3)}$ & T$_{tens.}$ & $E_{L,T}$ & abs. err.  \\
 \hline
$32^3$ &   $250$, ($15.8$) & $1.5$ &$1.5 \cdot 10^7$ & $1.5 \cdot 10^{-9}$  \\
 \hline
$48^3$ &   $3374$, ($58.8$) & $2.8$ &$1.12 \cdot 10^8$ & 0  \\
 \hline
$64^3$ &   -- & $5.7$ & $5.0  \cdot 10^8$ & --  \\
 \hline
$128^3$ &   -- & $13.5$ &$1.6 \cdot 10^{10}$ & --  \\
 \hline
$256^3$ &  -- & $68.2$ &$5.2 \cdot 10^{11}$ & --  \\
 \hline
 \end{tabular}
\end{center}%
\caption{\small Comparison of calculation times for computation of the interaction energy 
for increasing cluster size.}
\end{table}

\section{Tensor representation of multi-particle potentials}
\label{sec:Gener_Mpaticle_pot}

\subsection{How to treat systems of generally distributed particles} 
 \label{ssec:Pot_Generic}

 What if a many-particle system is not of lattice-type, but is unstructured like a protein?
 Then summation of many canonical tensors  
 \begin{equation} \label{eqn:P0_sum}
P_0(x) \mapsto  {\bf P}_0 = \sum_{\nu=1}^{N_0} {z_\nu}\, {\cal W}_\nu (\widehat{\bf P}_R).
\end{equation}  
 yields large canonical ranks, like in direct tensor summation  (see Section \ref{ssec:direct}) 
 when the number of   particles 
 increases (''a non-structured system``).
 \begin{figure}[tbh] \label{fig:Vilin_prot}
\centering   
\includegraphics[width=5.0cm]{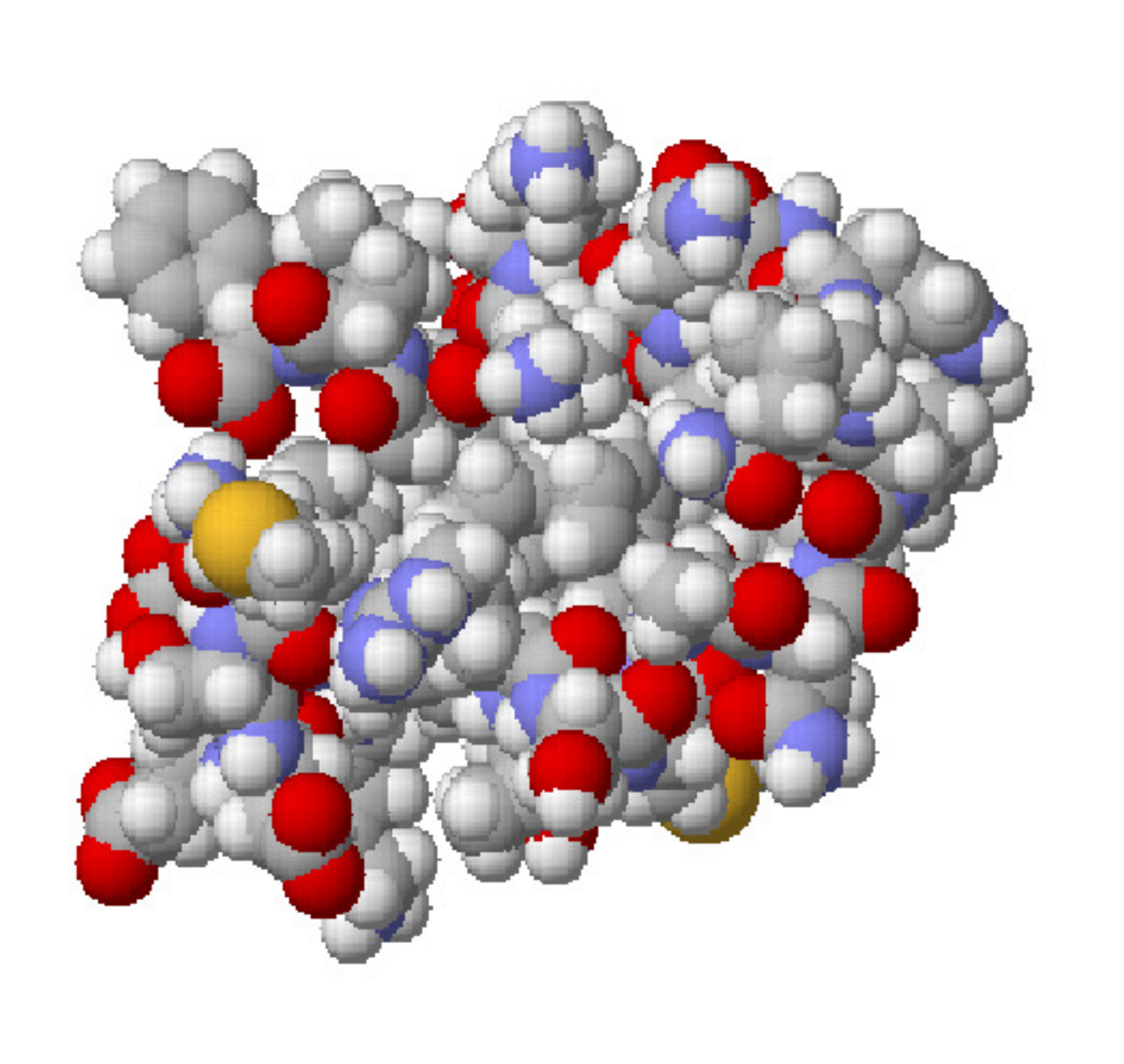}\quad \quad
\includegraphics[width=7.0cm]{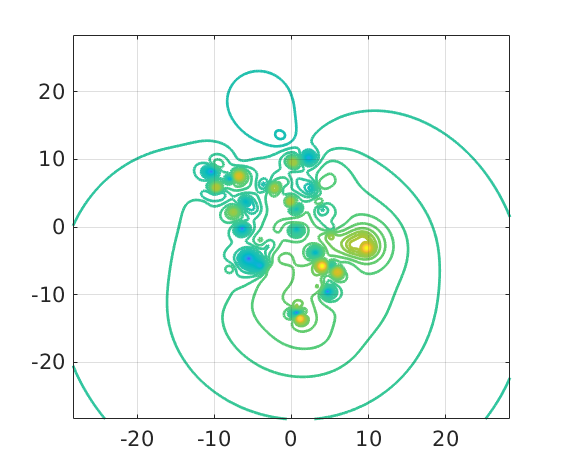} 
\caption{\small Left: a typical structure of  moderate size   protein molecule; 
right: the equi-potential contours of the total electrostatic potential of a protein-type 
molecule with 200 nuclei visualized at  the plane with $z=0$. }
\end{figure}
 
 Figure \ref{fig:Vilin_prot}, left demonstrates the structure of a villin protein \cite{Protein_data_bank:77} 
 and the   level set for the collective electrostatic potential.
  We notice that at the lower values of potential the contours are smooth and unite large groups 
 of particles, while at the vicinity of singularities we observe high density of levels. 
  This observation advocates the idea to separate the low-level part 
 of the collective potential.
 
 Figure \ref{fig:svd_prot} illustrates  
 the mode-$1$ singular values of the RHOSVD of the side matrices
 for the full potential of the  protein-type system (left) and of its low-range parts (right)
 versus the number of particles $N_0=200,400,774$ , $n=1024$. 
 \begin{figure}[tbh] 
\centering   
\includegraphics[width=7.0cm]{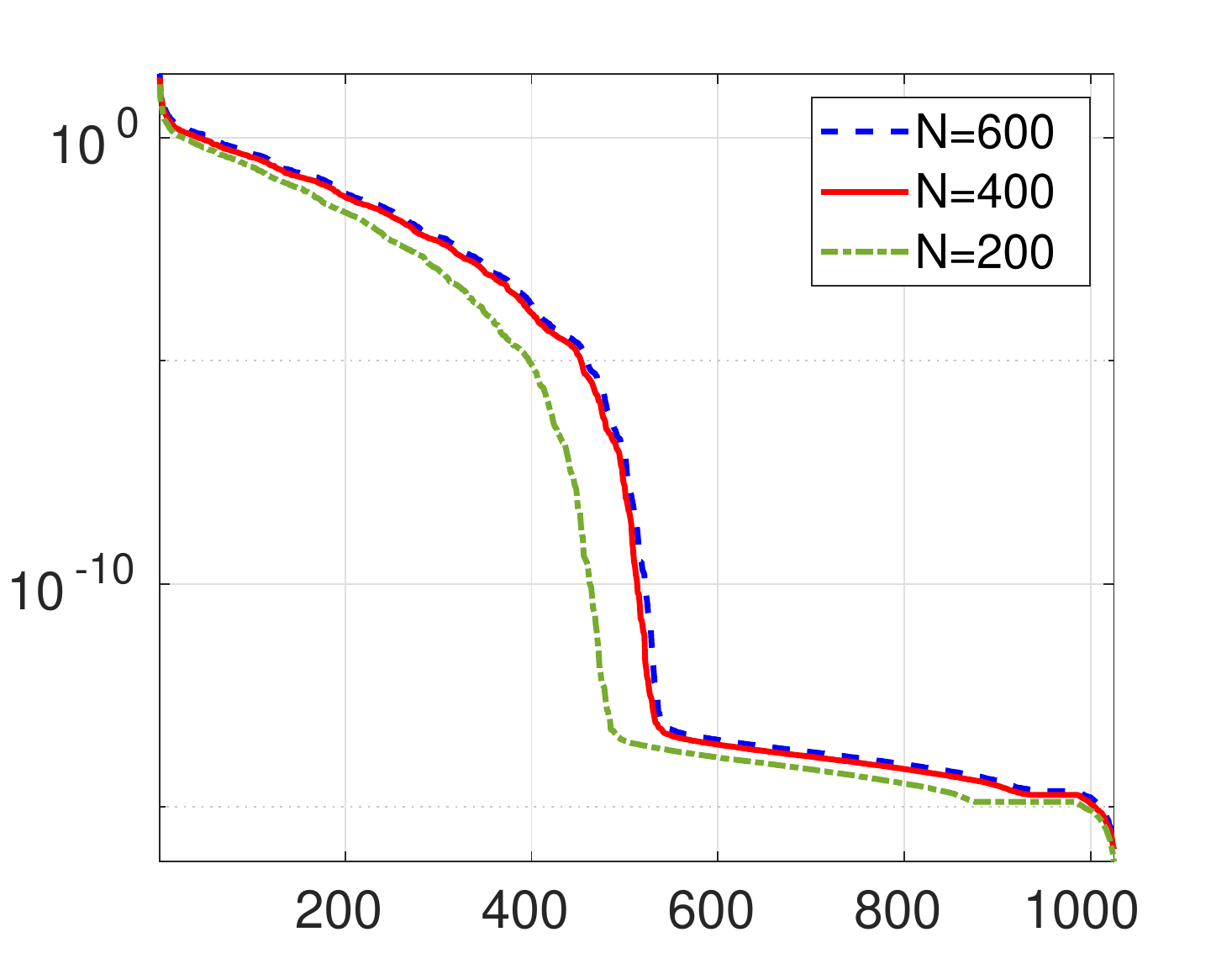}\quad \quad
\includegraphics[width=7.0cm]{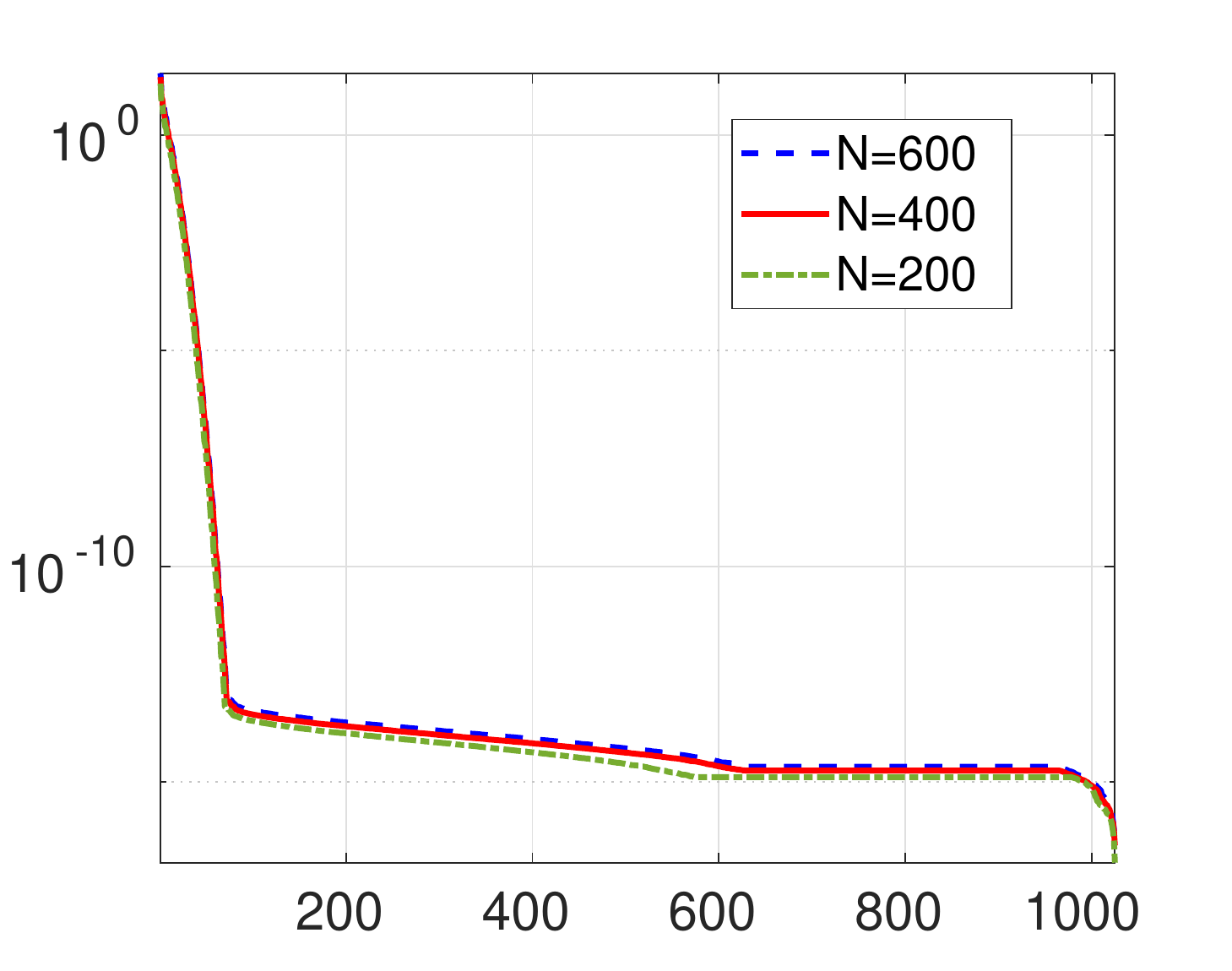} 
\caption{\small The singular values of the  mode-$1$ RHOSVD for the large sum of Newton kernels. }
\label{fig:svd_prot}
\end{figure}
 
This example indicates that in the case of a non-structured location of particles 
the Tucker/canonical rank of the corresponding grid-based tensor representation
of the collective potential increases proportionally to the number of particles. 
The large ranks make the tensor decomposition of the full electrostatic potential of a  
the biomolecules useless. 
The application of the range separated tensor decompositions allows to get rid of this bottleneck.

 \subsection{Range-separated tensor splitting of the generating kernel}\label{ssec:S_L_split}

 From the detailed consideration of the quadrature (\ref{eqn:sinc_general}), 
 we can observe
that the full set of approximating Gaussians includes two classes of functions: those with 
small "effective support" and the long-range functions. Consequently, functions from different classes 
may require different tensor-based schemes for their efficient numerical treatment.
Hence, the idea of the new approach is the constructive implementation of a 
range separation scheme that allows the independent efficient treatment 
of both the long- and short-range parts in each summand in (\ref{eqn:Electrost_sum}).

In what follows, without loss of generality, we confine ourselves to the case of 
the Newton kernel in $\mathbb{R}^3$, 
that is represented by the canonical sum in (\ref{eqn:sinc_general})
with $k=0,1,\ldots,M_0$, $M_0=2M+1$.
We observe that the sequence of quadrature points $\{t_k\}$,   
can be split into two subsequences, 
$
{\cal T}:=\{t_k|k=0,1,\ldots,M_0\}={\cal T}_l \cup {\cal T}_s,
$ 
with
\begin{equation} \label{eqn:Split_Qpoints}
{\cal T}_l:=\{t_k\,|k=0,1,\ldots,R_l\}, \quad  
\mbox{and} \quad {\cal T}_s:=\{t_k\,|k=R_l+1,\ldots,M_0\}.
\end{equation}
Here
${\cal T}_l$ includes quadrature points $t_k$ condensed ``near'' zero 
(more precisely, in the interval $(0,1]$), hence generating the long-range Gaussians (low-pass filters). 
In turn, ${\cal T}_s$ accumulates the increasing in $M\to \infty$ 
sequence of ``large'' sampling points $t_k$ (more precisely, in the interval $(1,\infty)$),
with the upper bound $C_0^2 \log^2(M)$, 
corresponding to the short-range Gaussians (high-pass filters).
We further denote ${\cal K}_l:=\{k\,| t_k\in {\cal T}_l \}$ and 
${\cal K}_s:=\{k\,| t_k \in {\cal T}_s \}$. 

Splitting (\ref{eqn:Split_Qpoints}) generates the additive decomposition of the canonical tensor
$\mathbf{P}_R$ onto the short- and long-range parts,
\[
 \mathbf{P}_R = \mathbf{P}_{R_s} + \mathbf{P}_{R_l},
\]
where
\begin{equation} \label{eqn:Split_Tens}
    \mathbf{P}_{R_s} =
\sum\limits_{k\in {\cal K}_s} {\bf p}^{(1)}_k \otimes {\bf p}^{(2)}_k \otimes {\bf p}^{(3)}_k, 
\quad \mathbf{P}_{R_l} =
\sum\limits_{k\in {\cal K}_l} {\bf p}^{(1)}_k \otimes {\bf p}^{(2)}_k \otimes {\bf p}^{(3)}_k.
\end{equation}

The choice of the critical number $R_l=\# {\cal T}_l -1$ (or equivalently, $R_s=\# {\cal T}_s = M-R_l$), 
that specifies the splitting ${\cal T}={\cal T}_l \cup {\cal T}_s$,
is determined by the \emph{active support} 
of the short-range components 
such that one can cut off the functions ${\bf p}_k(x)$, $t_k\in {\cal T}_s$, 
outside of the sphere $B_{\sigma}$ of radius ${\sigma}>0$, 
subject to a certain threshold $\delta>0$.
Fixed $\delta>0$, the choice of $R_s$ is  defined by 
the (small) parameter $\sigma$ and vise versa, see  \cite{BKK_RS:17} for further details.

For example, in electronic structure calculations, the parameter $\sigma$ can be associated with the 
typical inter-atomic distance in the molecular system of interest.
\begin{figure}[htb]
\centering
\includegraphics[width=7.0cm]{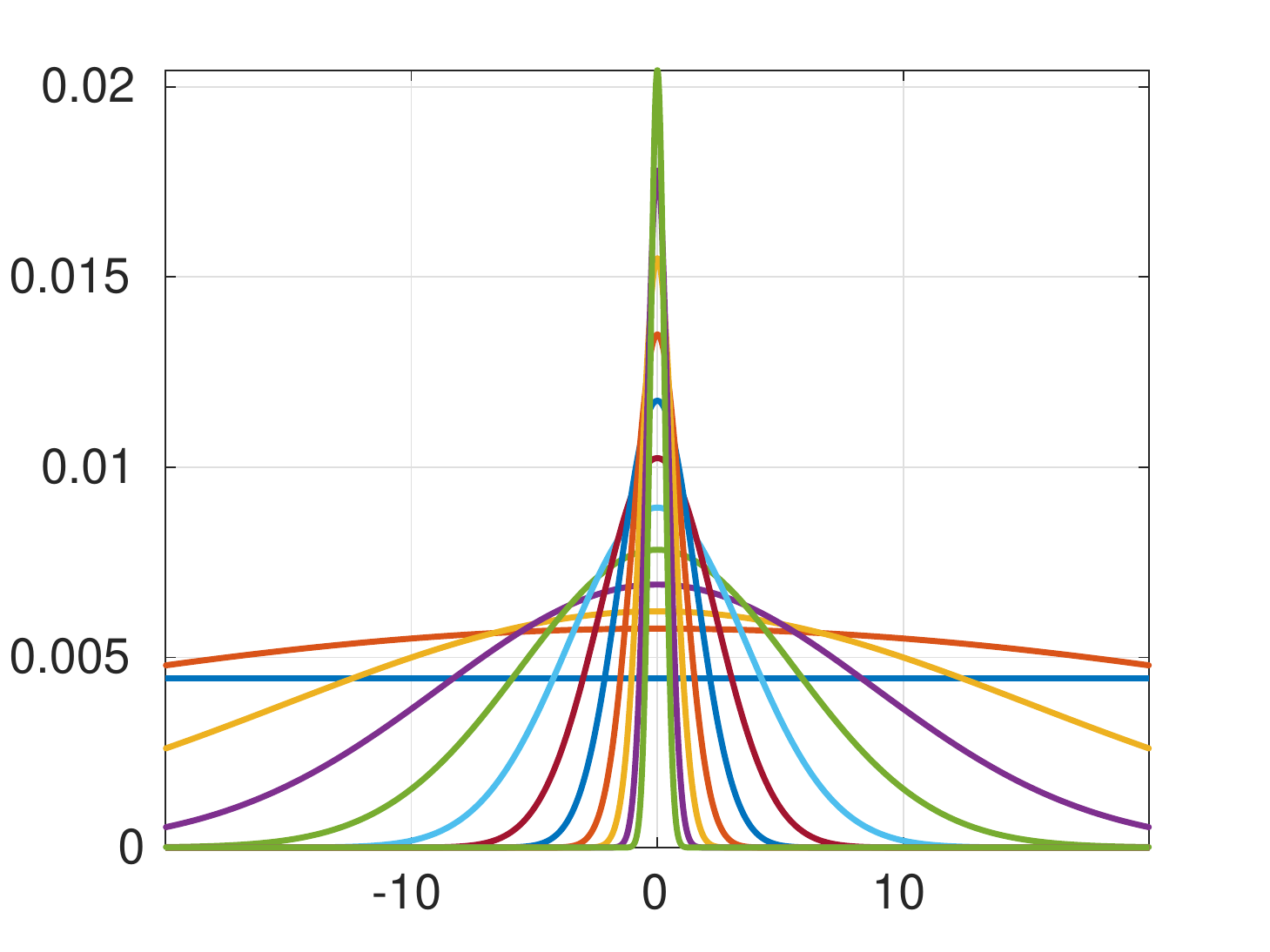} \quad
\includegraphics[width=7.0cm]{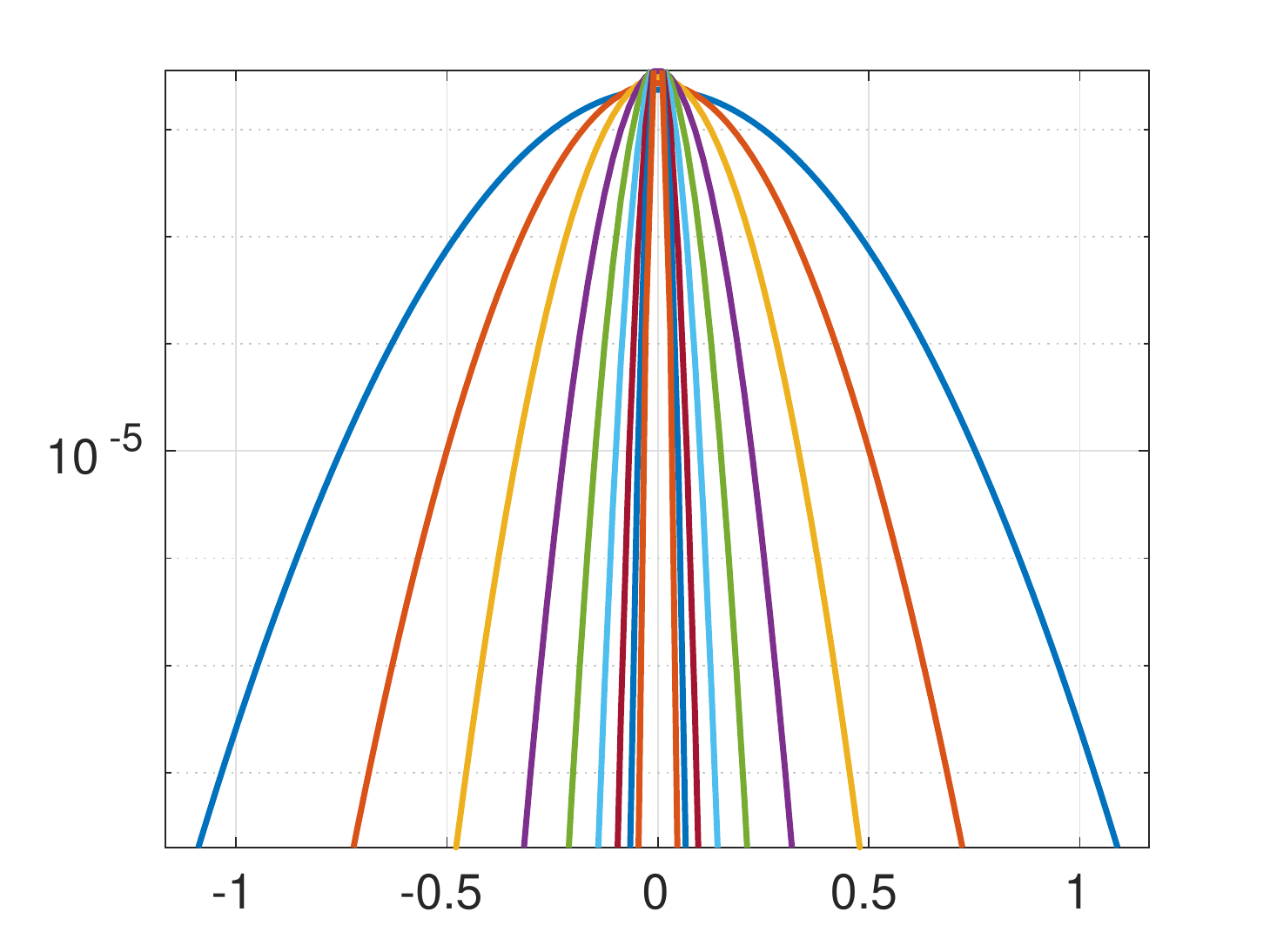}
 \caption{\small Long-range vectors (left) and short-range (in logarithmic scale) 
canonical vectors for the Newton kernel with $n=1024$, $R=24, R_l=13$. }
\label{fig:1024_rl12}
\end{figure}

Figure \ref{fig:1024_rl12}   
illustrates the splitting (\ref{eqn:Split_Qpoints}) for the tensor 
${\bf P}_R= {\bf P}_{R_l} + {\bf P}_{R_s}$ in (\ref{eqn:Split_Tens})
represented on 
the $n\times n\times n$ grid with the parameters $R=20, R_l=12$ and $R_s=8$, 
respectively (cf. \cite{BKK_RS:17}).
The canonical vectors of the discretized Newton kernel depicted in Figure \ref{fig:1024_rl12}, left,
exhibit long-range behavior, 
while the corresponding vectors of the tensor ${\bf P}_{R_s}$ decay exponentially fast 
apart of the effective support, see Figure \ref{fig:1024_rl12}, right.

\begin{figure}[htb]
\centering
\includegraphics[width=7.0cm]{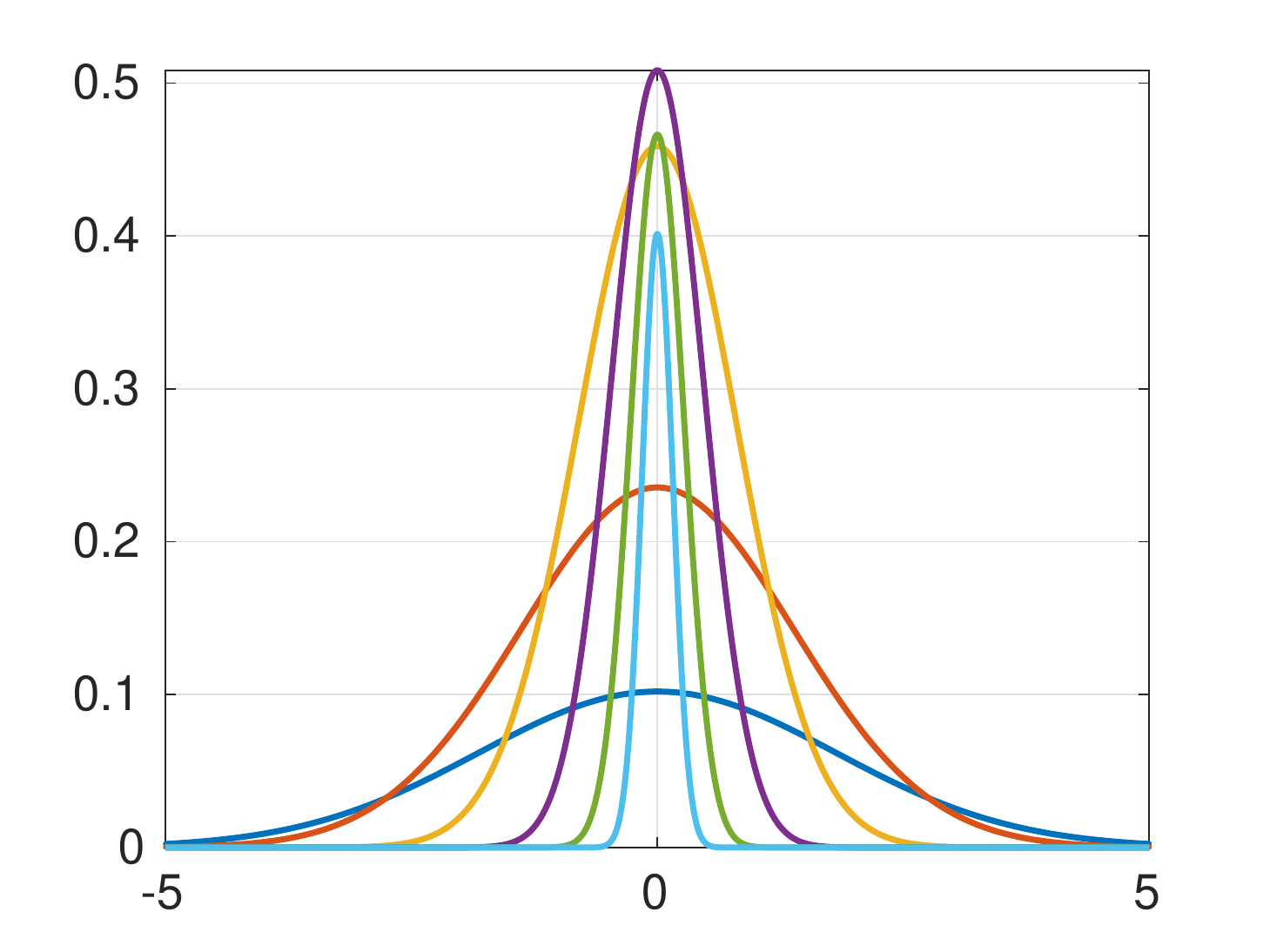} \quad
\includegraphics[width=7.0cm]{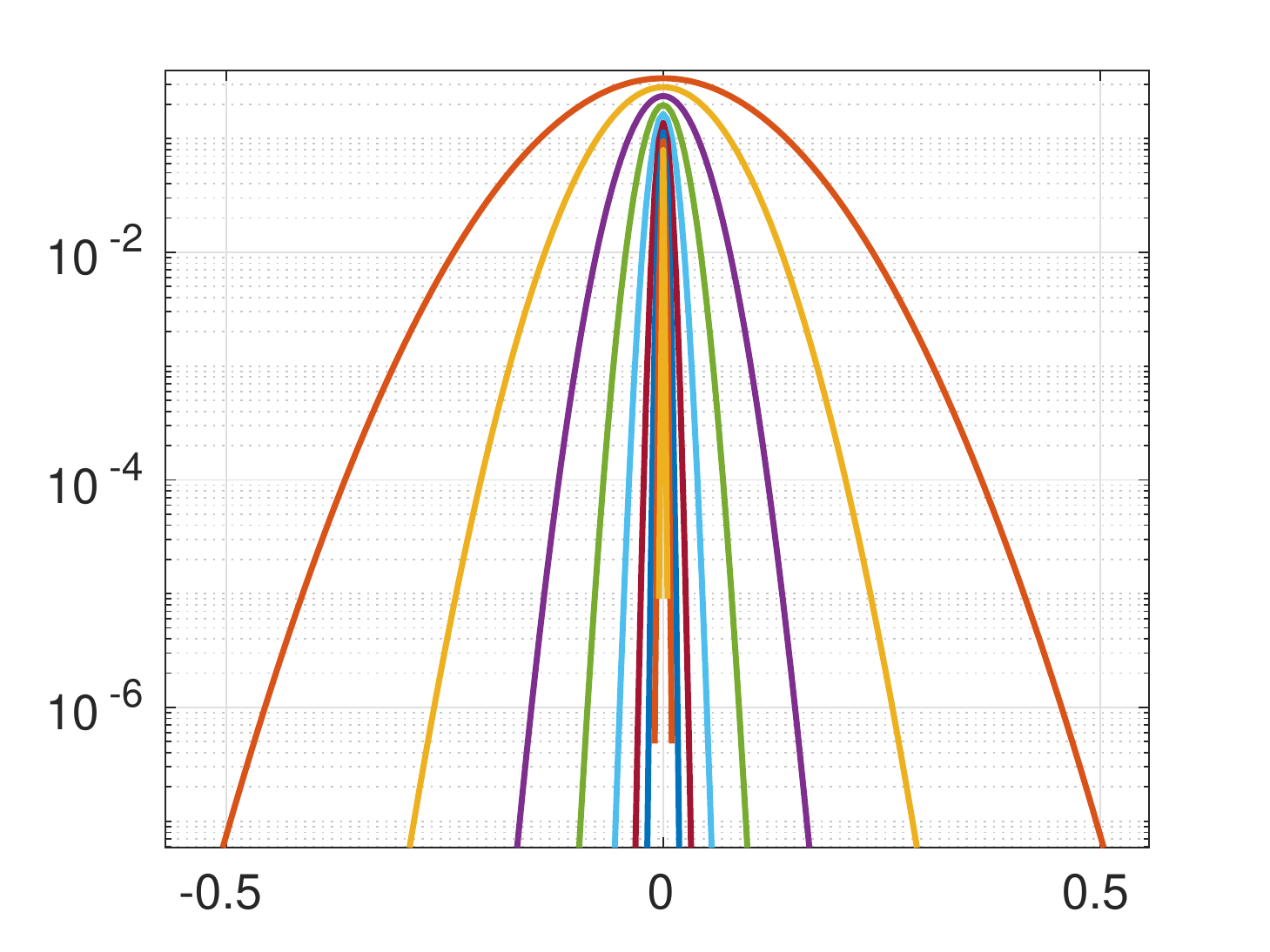}
 \caption{\small Long-range (left) and short-range (right) 
mode-1 canonical vectors for the 3D Slater function with $n=1024$, $R=24, R_l=6$. 
Short-range vectors are shown in logarithmic scale. }
\label{fig:1024_rl6_Slat}
\end{figure}

Figure \ref{fig:1024_rl6_Slat} illustrates the splitting (\ref{eqn:Split_Qpoints}) for the tensor 
${\bf P}_R= {\bf P}_{R_l} + {\bf P}_{R_s}$ in (\ref{eqn:Split_Tens})
representing the Slater potential $P(x)= e^{-\|x\|}$
on the $n\times n\times n$ grid with the parameters $R=24, R_l=6$ and $R_s=18$.
Notice that for this radial function the long-range part (Figure \ref{fig:1024_rl6_Slat}, left)
includes much less canonical vectors compared with the case of Newton kernel.
This anticipates the smaller total canonical rank for the long-range part 
in the large sum of Slater-like potentials arising, for example, in the representation 
of molecular orbitals and the electron density 
in electronic structure calculations. For instance, the wave function for the Hydrogen atom is 
given by the Slater function $e^{-\mu \|x\|}$, $x\in\mathbb{R}^3 $.

The advantage of the range separation in the splitting of 
the canonical tensor 
in  (\ref{eqn:Split_Tens}) is due to the opportunity for 
independent tensor representations of both sub-tensors ${\bf P}_{R_s}$ and ${\bf P}_{R_l}$ 
providing the separate grid-based treatment of the short- and long-range parts in the total sum of 
many pointwise interaction potentials as in (\ref{eqn:Electrost_sum}).

 \subsection{Brief introduction to the RS tensor format}   \label{ssec:RS_format}
 
 The novel range separated (RS) tensor format has been recently introduced in 
 \cite{BKK_RS:17}.
 This format is well suited for modeling of
 the long-range interaction potential in multi-particle systems, as well as for the 
 low-parametric interpolation of the scattered data. 
  It is based on the partitioning of the tensor representation of the reference kernel
  (usually the radial basis function)  into long- and short-range parts. 
  
 First, we recall the general definition of the RS tensor format
 \begin{definition}\label{Def:RS-Can_format} (RS-canonical tensors 
 \cite{BKK_RS:17}). 
 Given the separation parameter $\gamma \in \mathbb{N}$ and a set of points $x_\nu \in \mathbb{R}^{d}$,
 $\nu=1,\ldots,N$,
 the RS-canonical tensor format specifies the class of $d$-tensors 
 ${\bf A}  \in \mathbb{R}^{n_1\times \cdots \times n_d}$
 which can be represented as a sum of a rank-${R}_L$ canonical tensor  
 \begin{equation}\label{eqn:RL_Can}
{\bf A}_{R_L} = {\sum}_{k =1}^{R_L} \xi_k {\bf a}_k^{(1)} \otimes \cdots \otimes {\bf a}_k^{(d)}
\in \mathbb{R}^{n_1\times ... \times n_d}
\end{equation}
and a cumulated canonical tensor 
 $
 \widehat{\bf A}_S={\sum}_{\nu =1}^{N} c_\nu {\bf A}_\nu ,   
$
generated by replication of the reference tensor ${\bf A}_0$ to the predefined points 
$x_\nu$, $\nu=1,\ldots,N$.
Then the RS canonical tensor is parametrized by
\begin{equation}\label{eqn:RS_Can}
 {\bf A} =  {\bf A}_{R_L} + \widehat{\bf A}_S=
 {\sum}_{k =1}^{R_L} \xi_k {\bf a}_k^{(1)}  \otimes \cdots \otimes {\bf a}_k^{(d)} +
 {\sum}_{\nu =1}^{N} c_\nu {\bf A}_\nu, 
\end{equation}
where $\mbox{rank}({\bf A}_0)\leq R_0$ and $\mbox{diam}(\mbox{supp}{\bf U}_0)\leq 2 \gamma$ 
in the index size.
\end{definition}
The storage size for  the  RS-canonical tensor ${\bf A}$ in (\ref{eqn:RS_Can}) 
is estimated by (\cite{BKK_RS:17}, Lemma 3.9),
\[
\mbox{stor}({\bf A})\leq d R n + (d+1)N + d R_0 \gamma.
\]

The RS-Tucker tensor format is defined similar to Definition \ref{Def:RS-Can_format}
such that the rank-$R_L$ canonical tensor (\ref{eqn:RL_Can}) is substituted by the 
rank-$(r_1,\ldots,r_d)$ Tucker tensor.

The RS-canonical and Tucker tensor formats provide the powerful tool for data-sparse rank-structured 
representation of many-particle electrostatic potentials. In what follows we consider the 
basic example of the free space electrostatic potential of large system of charged particles.

According to the tensor canonical representation of the Newton kernel (\ref{eqn:sinc_general}) 
as a sum of Gaussians, one can distinguish their supports into the short- and long-range parts,
$
 \mathbf{P}_R = \mathbf{P}_{R_s} + \mathbf{P}_{R_l},
$
given by (\ref{eqn:Split_Tens}).
Then the RS splitting (\ref{eqn:Split_Tens}) is 
applied to the reference canonical tensor ${\bf P}_R$ 
and to its accompanying version $\widetilde{\bf P}_R=[\widetilde{p}_R(i_1,i_2,i_3)]$, 
$i_\ell \in \widetilde{I}_\ell$, $\ell=1,2,3$ living on the double size grid with the 
same mesh width, such that
\[
 \widetilde{\bf P}_R = \widetilde{\mathbf{P}}_{R_s} + \widetilde{\mathbf{P}}_{R_l} 
 \in \mathbb{R}^{2n \times 2n \times 2n}.
\]

  The total electrostatic potential $P_0(x)$ in (\ref{eqn:P0_sum}) is represented by a projected tensor 
${\bf P}_0\in \mathbb{R}^{n \times n \times n}$ that can 
be constructed by a direct sum of shift-and-windowing transforms of the reference 
tensor $\widetilde{\bf P}_R$ (see \cite{KhKh_CPC:13} for more details),
\begin{equation}\label{eqn:Total_Sum}
 {\bf P}_0 = \sum_{\nu=1}^{N} {z_\nu}\, {\cal W}_\nu (\widetilde{\bf P}_R)=
 \sum_{\nu=1}^{N} {z_\nu} \, {\cal W}_\nu (\widetilde{\mathbf{P}}_{R_s} + \widetilde{\mathbf{P}}_{R_l})
 =: {\bf P}_s + {\bf P}_l.
\end{equation}
The shift-and-windowing transform ${\cal W}_\nu$ maps a reference tensor 
$\widetilde{\bf P}_R\in \mathbb{R}^{2n \times 2n \times 2n}$ onto its sub-tensor 
of smaller size $n \times n \times n$, obtained by first shifting the center of
the reference tensor $\widetilde{\bf P}_R$ to the grid-point $x_\nu$ and then restricting 
(windowing) the result onto the computational grid $\Omega_n$.
  
The disadvantage of the tensor representation   (\ref{eqn:Total_Sum})  is
  that the number of terms in the canonical representation of the full tensor sum ${\bf P}_0$
increases almost proportionally to the number $N$ of particles in the system.

The remedy was found in \cite{BKK_RS:17} by considering  the global tensor 
decomposition of only the  "long-range part" in the tensor ${\bf P}_0$, defined by
\begin{equation}\label{eqn:Long-Range_Sum} 
 {\bf P}_l = \sum_{\nu=1}^{N} {z_\nu} \, {\cal W}_\nu (\widetilde{\mathbf{P}}_{R_l})=
 \sum_{\nu=1}^{N} {z_\nu} \, {\cal W}_\nu 
 (\sum\limits_{k\in {\cal K}_l} \widetilde{\bf p}^{(1)}_k \otimes \widetilde{\bf p}^{(2)}_k 
 \otimes \widetilde{\bf p}^{(3)}_k).
 \end{equation}
 The tensor representation of the sum of short-range parts is given by a sum of cumulative 
 tensors of small support (and small size), accomplished by the list of the 3D
 potentials coordinates.

 In application to the calculation of multi-particle interaction potentials  discussed above
 we associate the tensors ${\bf P}_s$ and ${\bf P}_l$ in (\ref{eqn:Total_Sum}) 
 with short- and long-range components ${\bf A}_{R_L}$ and  $\widehat{\bf A}_S$ in the RS 
 representation of the collective electrostatic potential ${\bf P}_0$.
 The following theorem proofs the almost uniform in $N$ bound on the Tucker (and canonical) 
 $\varepsilon$-rank of the tensor ${\bf A}_{R_L}={\bf P}_l$, representing the long-range part of ${\bf P}_0$.
  \begin{theorem}\label{thm:Rank_LongRange}
 (Uniform rank bounds for the long-range part,  \cite{BKK_RS:17}).
Let the long-range part ${\bf P}_l$ in the total interaction potential, see (\ref{eqn:Long-Range_Sum}),
correspond to the choice of splitting parameter 
in (\ref{eqn:Split_Qpoints}) with $M_0=O(\log^2\varepsilon)$.
Then the total $\varepsilon$-rank ${\bf r}_0$ of the Tucker approximation to the canonical tensor sum ${\bf P}_l$
is bounded by
\begin{equation}\label{eqn:Rank_LongR}
 {\bf r}_0:=\mbox{rank}_{Tuck}({\bf P}_l)=C\, b \,\log^{3/2} (|\log (\varepsilon/N)|),
\end{equation}
where the constant $C$ does not depend on the number of particles $N$, 
as well as on the size of the computational box, $[-b,b]^3$.
\end{theorem}

The estimate (\ref{eqn:Rank_LongR}) indicates that for fixed size of computational box embedding 
the many-particle system the numerical cost for the canonical approximation of the long range part in 
the collective potential depends only logarithmically on the approximation accuracy 
and the number of particles.

\begin{corollary} \label{cor:Rank_longR_CP_TT}
 The simple consequence of Theorem \ref{thm:Rank_LongRange} is the following estimate 
 on the CP rank and the TT-rank of the long-range part in the collective electrostatic potential 
 ${\bf P}_l$,
 \[
  \mbox{rank}_{CP}({\bf P}_l)\leq {\bf r}_0^2, \quad \mbox{rank}_{TT}({\bf P}_l)\leq {\bf r}_0^2,
 \]
 where the Tucker rank ${\bf r}$ is bounded by (\ref{eqn:Rank_LongR}).
\end{corollary}

Numerical rank reduction in the sum (\ref{eqn:Long-Range_Sum}) with the initial canonical rank
$R_0=R N$ is performed by the  C2T
algorithm that is based on the SVD of the $n\times R_0$ side matrices in the canonical 
sum (\ref{eqn:Long-Range_Sum}).
In case of very large number of particles $N$ the C2T algorithm can be applied 
in the ``add-and-compress'' version. In this approach the initial group of $N$ particles is 
decomposed into certain number $m_0$ of small subgroups such that the long-range part is precomputed for  
 each subgroup and then rank compression applies to the sum of $m_0$ canonical tensors with 
 a small rank.  The optimal number $m_0$ can be easily estimated by solving simple optimization 
 problem for the cost functional.

 \section{RS tensor format in bio-molecular modeling} \label{sec:rs-format}
 
 \subsection{Interaction energy and forces via long-range part in the potential} 
 \label{ssec:rsPot_Energ_Forces}

 The RS tensor format introduced in \cite{BKK_RS:17} applies to many particle 
systems with rather generally located potentials all discretized on the fine tensor grid in $\mathbb{R}^3$. 
These can be the electrostatic potentials 
of large atomic systems like bio-molecules or the multidimensional scattered data modeled by
radial basis functions. The main advantage of the RS tensor format is that
the partition into the long- and short-range parts is performed just by sorting skeleton vectors
in the tensor representation of the generating kernel,
\[
  \mathbf{P}_R = \mathbf{P}_{R_s} + \mathbf{P}_{R_l}.
\]   
Theorem \ref{thm:Rank_LongRange} proves that the sum of long range contributions 
from all particles in the collective
potential is represented by a low-rank canonical/Tucker tensor 
with a rank which only logarithmically depends on the number of particles $N$. 
The representation complexity of the short range part is $O(N)$ with a small constant 
independent on the number of particles.

The basic tool for fast calculation of the long-range part is  the canonical-to-Tucker algorithm 
combined with the reduced higher order SVD (RHOSVD), see \cite{khor-ml-2009}. 
  The error of the RS tensor representation is defined by the
$\varepsilon$-truncation threshold for the Tucker-to-canonical algorithm in the rank reduction scheme.
Figure \ref{fig:long} illustrates the cross-sections of the collective 
free space electrostatic potential for $650$ atomic biomolecule \cite{BKKKM_PBE:18} 
computed by the RS tensor representation. 
\begin{figure}[tbh]
\centering  
\includegraphics[width=7cm]{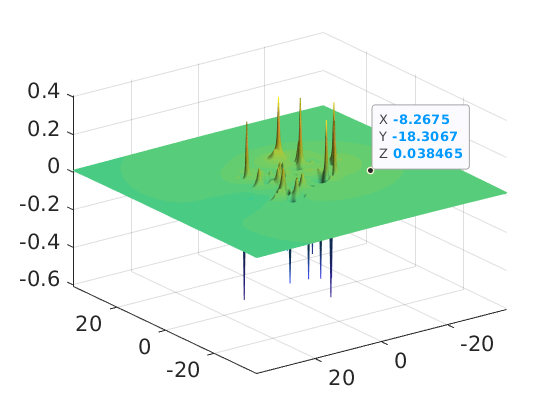}\quad
 \includegraphics[width=7cm]{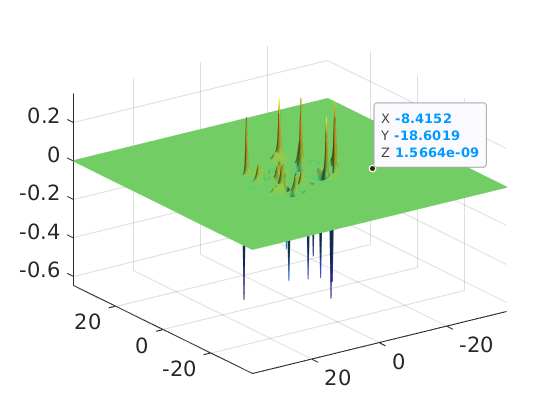}\\
 \includegraphics[width=7cm]{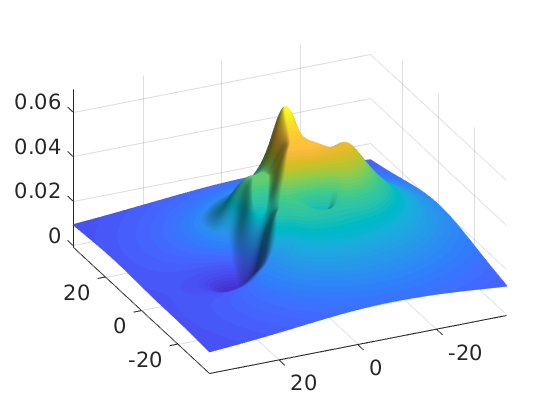}\quad
 \includegraphics[width=7cm]{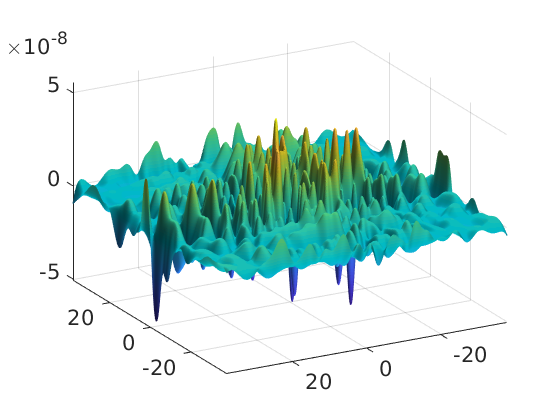}
 \caption{\small Top:
Left: the free space electrostatic potential of a small biomolecule computed by RS
tensor format. Right: the short-range part of the RS tensor.
Bottom:
The long-range part of the collective free space electrostatic 
potential of a small biomolecule (left) and the error of RS representation (right). } 
\label{fig:long}  
\end{figure}

The beneficial feature of the RS decomposition is due to the low-parametric grid-based 
 representation of the long-range component in the many-particle potential which allows 
 to recover the most important physical characteristic of the system 
 such that the interaction energy and  forces.  
 Recall that the electrostatic interaction energy is represented in the form
  \begin{equation}
 \label{inter_energy} 
 E_{N}= E_N(x_1,\dots,x_N)= \sum^{N}_{i=1} \sum^{N}_{j <i} \frac{z_i z_j}{\|x_i- x_j\|},  
 \end{equation}
 and it can be computed by direct summation in $O(N^2)$ operations.
 Define vectors ${\bf z}=(z_1,\ldots,z_N)^T\in \mathbb{R}^N$ and 
 ${\bf p}_l=({\bf P}_l({x}_{1}), \ldots,{\bf P}_l({x}_{N})0^T\in \mathbb{R}^N $.
 Given the long-range tensor representation ${\bf P}_l$,
 the following  lemma represents the electrostatic interaction energy by simple operation with
 the vectors ${\bf z}$  and ${\bf p}_l$ which can be precomputed at $O(N)$ cost up to lower order terms. 
  The following statement is the minor modification of Lemma 4.2  in \cite{BKK_RS:17}. 

 \begin{lemma}\label{lem:InterEnergy} 
 Let the effective support of the short-range components in the reference
 potential ${\bf P}_R$ do not exceed $\sigma>0$.
 Then the interaction energy $E_N$ of the $N$-particle system  can be
 calculated by using only the long range part  in the total potential sum
 \begin{equation}\label{eqn:EnergyFree_Tensor}
 E_N = \frac{1}{2} 
 \sum\limits_{{j}=1}^N z_{j}({\bf P}_l({x}_{j}) - z_j {\bf P}_{R_l}(0))=
 \frac{1}{2}\langle {\bf z},{\bf p}_l \rangle - \frac{{\bf P}_{R_l}(0)}{2} \sum\limits_{{j}=1}^N z_j^2,
\end{equation}
in $O(d R_l N)$ operations, where $R_l$ is the canonical rank of the long-range component.
\end{lemma}
Here the term $\frac{1}{2}\langle {\bf z},{\bf p}_l \rangle$ denotes the "non-calibrated" 
interaction energy with the long-range tensor component ${\bf P}_l$.
Lemma \ref{lem:InterEnergy} indicates that the interaction energy does not depend on the short-range 
part in the collective potential which is the key point for the construction of
energy preserving regularized numerical schemes for solving the basic equations in bio-molecular 
modeling.

Notice that the energy expansion (\ref{eqn:EnergyLattice}) is the particular version of 
(\ref{eqn:EnergyFree_Tensor}) in the case of lattice structured system with equal charges.

The other important application of the RS format is due to the opportunity to recompute gradients and
the force field at each particle location of interest in applications to multi-particle dynamics.
In particular, such computational tasks arise in the problem of protein-ligand docking, 
in the process of electrostatic self-assembly, and for the general use in many-particle classical dynamics.


 Calculation of electrostatic forces and gradients of the interaction potentials 
 in multiparticle systems is a computationally extensive problem.
The algorithms based on Ewald summation technique were discussed in \cite{DesHolmII:98,HoEa:88}.
We propose an alternative approach using the RS tensor format. 

First, we consider the gradients operator.
Given an RS-canonical tensor ${\bf A}$ as in (\ref{eqn:RS_Can}) with the width parameter $\gamma >0$, 
the discrete gradient $\nabla_h=(\nabla_1,\ldots,\nabla_d)^T$ applied to the long-range
part in ${\bf A}={\bf A}_s + {\bf A}_l$ at all grid points of $\Omega_h$ simultaneously,  
can be calculated as the $R$-term canonical tensor by applying the simple
one-dimensional finite-difference (FD) operations to each rank-$1$ term in the long-range 
part  ${\bf A}_l$,
\begin{equation}\label{eqn:Gradients_Tens}
\nabla_h {\bf A}_l= {\sum}_{k =1}^{R} \xi_k ({\bf G}_k^{(1)},\ldots,{\bf G}_k^{(d)})^T,
\end{equation}
with rank-$1$ tensor entries
$
 {\bf G}_k^{(\ell)} ={\bf a}_k^{(1)}  \otimes \cdots\otimes 
 \nabla_\ell {\bf a}_k^{(\ell)} \otimes \cdots\otimes{\bf a}_k^{(d)},
$
where $\nabla_\ell$ ($\ell=1,\dots,d$) is the univariate FD differentiation scheme in variable 
$x_\ell$ (by using backward or central differences). Numerical complexity of
the representation (\ref{eqn:Gradients_Tens}) can be estimated by $O(d R n )$ provided that
the canonical rank is almost uniformly bounded in the number of particles. 
The gradient operator applies locally to each short-range term in (\ref{eqn:RS_Can})
which amounts in the complexity $O(d R_0 N)$.


Given the electrostatic potential energy $E_N$,
the force vector ${\cal F}_j$ on the particle $j$ is obtained by differentiating the 
functional  $E_N(x_1,\dots,x_N)$ with respect to $x_j$,
\[
 {\cal F}_j=-\frac{\partial}{\partial x_j} E_N = - \nabla_{| x_j} E_N,
\]
which can be calculated explicitly (see \cite{HoEa:88}) in the form,
\begin{equation} \label{eqn:Force_j_direct}
 {\cal F}_j = \frac{1}{2} z_{j} \sum\limits_{{k}=1, {k}\neq {j}}^N  z_{k} 
 \frac{{x}_{j} - {x}_{k} }{\|{x}_{j} - {x}_{k}\|^3}.
\end{equation}
The Ewald summation technique for force calculations on the positions of particles
was discussed for example  in \cite{DesHolmII:98,HoEa:88}.

As an alternative to the direct summation by (\ref{eqn:Force_j_direct}) we discuss two diffent 
approaches for fast evaluation of (\ref{eqn:Force_j_direct}) based on the 
the RS tensor representation.

(A) The RS tensor approximation can be applied directly to the discretized 
electric field on the potential $P_0$, that is a sum of dipole-like terms,
\[
 {\cal E}(x) =-\nabla P_0(x)=  \sum\limits_{{k}=1}^N  z_{k} 
 \frac{{x} - {x}_{k} }{\|{x} - {x}_{k}\|^3}.
 \]
 This approach will be considered elsewhere.
 
(B) Calculation the force field components ${\cal F}_j$ by the direct FD numerical differentiation of 
the energy functional $E_N$ by using RS tensor representation 
of the long-range part in the $N$-particle interaction potential $P_0(x)$ discretized on fine spacial grid
(see \cite{BKK_RS:17}). 

In approach (B),
the differentiation in RS-tensor format with respect to $x_j$
is based on the explicit representation (\ref{eqn:EnergyFree_Tensor})
which requires only the calculation of long-range tensor ${\bf P}_l$.
 We chose $j=N$ for example, and apply the FD differentiation to the energy 
representation (\ref{eqn:EnergyFree_Tensor}) to obtain
\[
 E_N(x_1,\dots,x_N) - E_N(x_1,\dots,x_N-h{\bf e}_i)=  
 \frac{1}{2}\langle {\bf z},{\bf p}_l(x_1,\dots,x_N) - {\bf p}_l(x_1,\dots,x_N-h{\bf e}_i)\rangle, 
\]
for three different values of the axis vectors
${\bf e}_1=(1,0,0)^T$, ${\bf e}_2=(0,1,0)^T$ and ${\bf e}_3=(0,0,1)^T$.
The implementation needs a simple recalculation of the smooth potential ${\bf P}_l$ under small variation 
of the only one particle center $x_N$ leading to the cost $O(d R n )$, uniformly in $N$. 

The energy and force field calculus discussed above are related to the free space electrostatic
potential. The case of bio-molecular simulations, the protein like structures are modeled by
solute-solvent systems in polarizable media, where substructures with different permeability 
constants are included.
The electrostatics in such systems is most commonly described by the linear 
(or nonlinear) Poisson-Boltzmann equation to be addressed in what follows.

 \subsection{The linearized Poisson-Boltzmann equation}
 \label{ssec:PBE_definition}

 The Poisson-Boltzmann equation (PBE) describes the electrostatic potential of  
 an ensemble of charged particles 
 immersed in a continuum dielectric medium with fixed permittivity $\epsilon_s$ 
 provided that each charge is embedded 
 into a small sphere with different from $\epsilon_s$ dielectric constant.
 The PBE can be formulated either as the elliptic partial differential 
 equation in $\mathbb{R}^3$ \cite{Holst:94,CaMaSt:13},  or  as the boundary integral 
 equation over the solute-solvent  interface \cite{CaMeTo:97,Maday:2018,XieYing:JCAM-2016} 
 (in the case of linear model in $\mathbb{R}^3$). The PBE model is supposed to capture 
 the polarization effects in many-body systems.
 
 The basic techniques on the analysis and  numerical 
 approximation of the boundary and volume integral equations have been addressed in 
 \cite{MazSch:book:07,HsiWendBOOK:08,SauterSchwa_book:11}. The commonly used techniques 
 for solving the arising 
 discretized elliptic equations in $\mathbb{R}^3$ are based on either multigrid 
 iteration \cite{Holst:94}  or domain decomposition methods \cite{LiStCaMaMe:13,CaMaSt:13}. 
 
  The linearized PBE in the differential form reads 
  as the elliptic boundary value problem   
  \begin{equation}\label{eqn:PBE}
 -\nabla\cdot(\epsilon \nabla u)+ \kappa^2 u=\rho_f:= 
 {\sum}_{{k}=1}^N z_k \delta({x} - {x}_{k})
 \quad \mbox{in } \quad \Omega=\Omega_s \cup \Omega_m, 
\end{equation}
where $u$ is the target electrostatic potential of a protein,  $\rho_f$
is the scaled singular charge distribution (here $\delta(x)$ is the Dirac delta) supported at points $x_k \in \Omega_m$ in the solute 
region, where $\epsilon=\epsilon_m$ and $\kappa =0$, while in the external solvent domain $\Omega_s$ we have $\kappa \geq 0,\epsilon=\epsilon_s$. 
The interface conditions arise from the continuity of the potential and flaxes
on $\Gamma=\partial \Omega_m$: 
$[u]_{\Gamma}=0, \; \left[ \epsilon\frac{\partial u}{\partial n}  \right]_\Gamma=0\; \mbox{on}\; \Gamma.$
Moreover, the trace of the solution on $\partial \Omega$ or some other quantities may be fixed.
  \begin{figure}
 \begin{center}
\includegraphics[width=6cm]{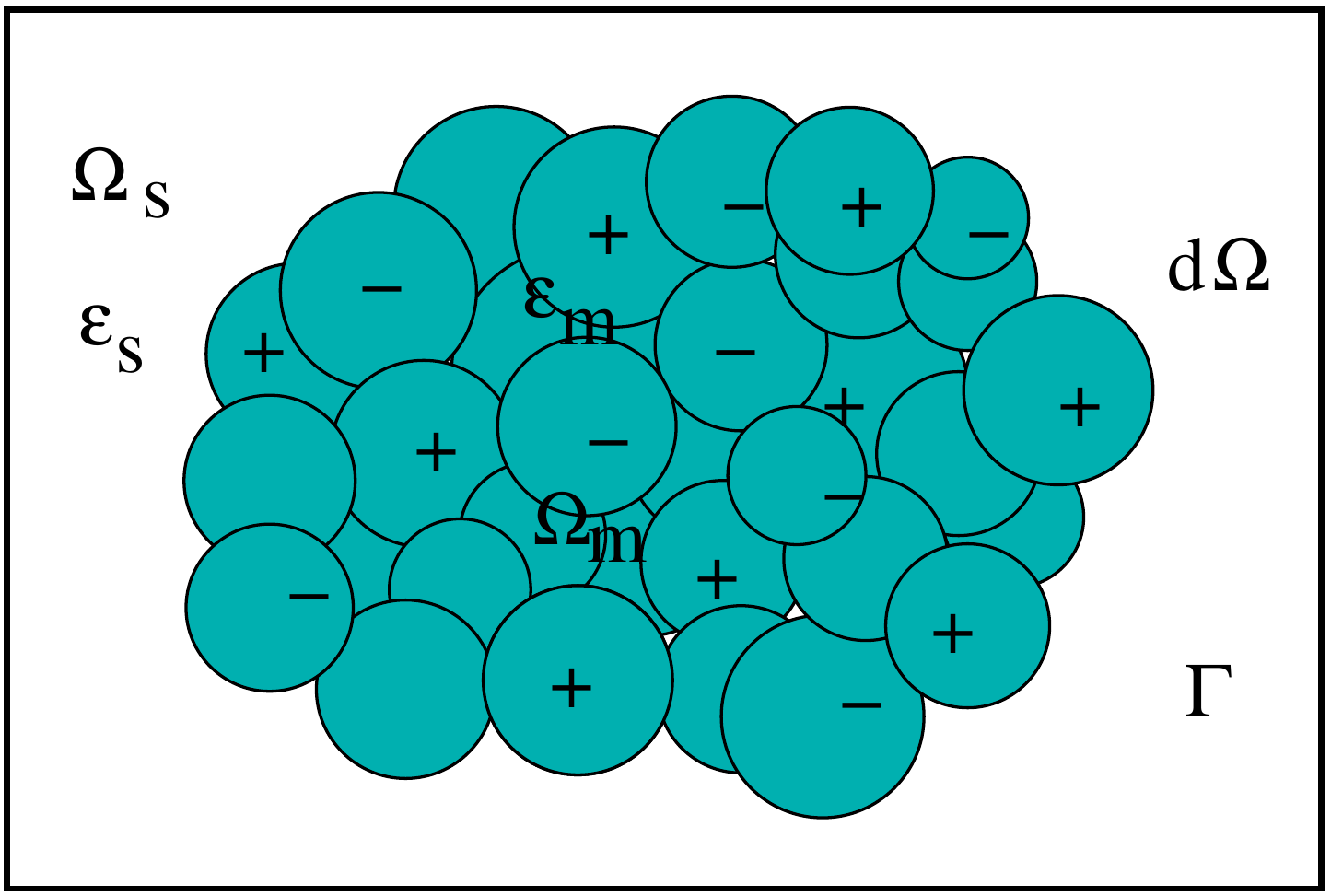} 
 \end{center}
 \caption{\small Computational domain for the PBE.}
 \label{fig:PBE_geometry}
  \end{figure}

The traditional FEM discretization methods could not be applied directly to the equation 
(\ref{eqn:PBE}) because the lack of regularity in the solution $u$. The traditional approaches 
are merely based on the regularization of this equation by subtraction of the exact free space 
electrostatic potential that solves the linear equation $-\Delta u_0 = \rho_f$ in $\mathbb{R}^3$.
Such approaches require the essential modification of the interface and boundary conditions,
as well as the non-trivial (especially in the nonlinear case) 
recovery procedure of the total interaction energy since the component $u_0$ may include 
considerable portion of energy.

In what follows we discuss the recently introduced new  {\it energy preserving}
regularization scheme based on (A) the RS splitting of the discretized Dirac 
delta \cite{BKhor_Dirac:18}, and (B) using the RS tensor format \cite{BKK_RS:17} 
for the construction of the regularized solution \cite{BKhor_Dirac:18,BeKhKhKwSt:18,BKKKM_PBE:18}.

 \subsection{The RS tensor decomposition of the discretized Dirac delta}
 \label{ssec:RS_delta}

 The RS tensor splitting of the discretized Dirac delta introduced in \cite{BKhor_Dirac:18}
 is based on the idea that discretized Dirac delta can be recovered from 
 the governing equation  
 \begin{equation}\label{eqn:DiracD_Newton}
 -\Delta \frac{1}{\|x\|} =4 \pi \delta(x), \quad x\in \mathbb{R}^3,   
\end{equation}
discretized on the fine tensor grid in the computational box in $\mathbb{R}^3$. 
Indeed, we can discretize all entities on the left hand side of the equation (\ref{eqn:DiracD_Newton})
via grid-based approximations which admit the low-rank representations as follows
\[
 \frac{1}{\|x\|}   \rightsquigarrow \mathbf{P}_R = \mathbf{P}_{R_s} + \mathbf{P}_{R_l} 
 \in \mathbb{R}^{n \times n \times n}, 
 \quad \mbox{and} \quad \Delta \rightsquigarrow {\Delta}_h,
\]
where ${\Delta}_h$ is the finite difference (FD) Laplacian over $n \times n \times n$ grid with 
mesh size $h$. In particular, we use the Kronecker rank-$3$ representation
\begin{equation}\label{eqn:Lapl_Kron3}
{\Delta}_h = \Delta_{1} \otimes I_2\otimes I_3 + I_1 \otimes \Delta_{2} \otimes I_3 + 
I_1 \otimes I_2\otimes \Delta_{3}. 
\end{equation}
This introduces the grid representation of the Dirac delta
\begin{equation} \label{eqn:Delta_discrete}
 \delta(x) \rightsquigarrow \boldsymbol{\delta}_h:=- \frac{1}{4 \pi}{\Delta}_h \mathbf{P}_R,
\end{equation}
which is associated with its particular differential representation (\ref{eqn:DiracD_Newton}).
Now the rank-$3R$ CP tensor representation $\boldsymbol{\delta}_h$ of the discretized 
Dirac delta ${\delta}$ approximated on 
$n\times n \times n$ Cartesian grid can be computed as the action of the 
Laplace operator on the Newton kernel given in the canonical rank-$R$ tensor format as follows
\[
 \boldsymbol{\delta}_h = -A_{\Delta} \mathbf{P}_R=
 \sum\limits_{k=1}^R ( \Delta_1 {\bf p}^{(1)}_k \otimes {\bf p}^{(2)}_k \otimes {\bf p}^{(3)}_k 
 +  {\bf p}^{(1)}_k \otimes \Delta_2 {\bf p}^{(2)}_k \otimes {\bf p}^{(3)}_k 
 +  {\bf p}^{(1)}_k \otimes  {\bf p}^{(2)}_k \otimes \Delta_3 {\bf p}^{(3)}_k). 
\]
The short- and long-range splitting of the 
 discretized $\boldsymbol{\delta}_h$ is defined by
 \[
  \boldsymbol{\delta}_h= \boldsymbol{\delta}_s + \boldsymbol{\delta}_l,
  \quad \mbox{where} \quad 
 \boldsymbol{\delta}_s:= - {\Delta}_h \mathbf{P}_{R_s}, \quad 
 \mbox{and} \quad \boldsymbol{\delta}_l:= - {\Delta}_h \mathbf{P}_{R_l}.
 \]
 
In the PBE setting,  the construction of short-range part $\mathbf{P}_{R_s}$ implies 
that $\boldsymbol{\delta}_s$ vanishes on the interface $\Gamma$, and hence the potential 
$\mathbf{P}_{R_s}$ satisfies the discrete Poisson equation in $\Omega_m$ with the right-hand
side in the form $\boldsymbol{\delta}_s$ and zero boundary conditions on $\Gamma$.
 
The above RS decomposition of the Dirac delta can be extended to the case of many-atomic
systems as in the case of PBE (\ref{eqn:PBE}), see \cite{BKhor_Dirac:18}.
 Lemma 3.1 in \cite{BKhor_Dirac:18} proves that the Tucker rank of the long-range part in the 
 $N$-particle discretized Dirac delta depends only logarithmically on the quantity 
 $\varepsilon/N$, e.g., 
 \[
 {\mbox{rank}}_{Tuck}({\delta}_{s,N})= O(\log^{3/2}(\varepsilon/N)).
 \]

Now we are in a position to describe the energy preserving tensor regularization scheme for PBE.

 \subsection{Energy preserving tensor regularization scheme for PBE}
 \label{ssec:RS_reg_PBE}
 
To avoid some technical details we describe  the energy preserving tensor 
regularization scheme for PBE (\ref{eqn:PBE}) on the continuous level, see \cite{BKhor_Dirac:18}.

Let $u_f=u_{\mbox{\footnotesize short}} + u_{\mbox{\footnotesize long}}$ 
be the RS splitting of the free space potential generated 
by the density $\frac{1}{\epsilon_m} f_\rho$, i.e., satisfying the equation 
\begin{equation}\label{eqn:freeSpaceRS}
-\epsilon_m \Delta u_f = f_\rho \quad   \mbox{in} \quad  \mathbb{R}^3.
\end{equation}
This introduces the 
corresponding RS decomposition of the density $f_\rho$ in (\ref{eqn:PBE}) 
 due to the equation
\begin{equation}\label{eqn:RS_decomp_Rho} 
 -\epsilon_m \Delta (u_{\mbox{\footnotesize short}} + u_{\mbox{\footnotesize long}}) 
 = f_\rho =  \rho_{\mbox{\footnotesize short}} + \rho_{\mbox{\footnotesize long}},
\end{equation}
where
\begin{equation}\label{eqn:Short_Dir_Newt_eqn}
 \rho_{\mbox{\footnotesize short}} = -\epsilon_m \Delta u_{\mbox{\footnotesize short}}, \quad
 \mbox{and} \quad 
 \rho_{\mbox{\footnotesize long}}=-\epsilon_m \Delta u_{\mbox{\footnotesize long}}.
\end{equation}
 The above RS decomposition suggests  the new regularization scheme 
by the RS splitting of the solution $u$  of (\ref{eqn:PBE})  in the form
\begin{equation}\label{eqn:PBE_regul_solution}
 u= \overline{u}_{\mbox{\footnotesize long}} + u_{\mbox{\footnotesize short}},
\end{equation}
where the component ${u}_{\mbox{\footnotesize short}}$ is precomputed explicitly 
and stored in the canonical/Tucker tensor format living on the fine Cartesian grid,  
and the unknown function $\overline{u}_{\mbox{\footnotesize long}}$ satisfies the equation 
with the modified right-hand side 
\begin{equation}\label{eqn:PBE_regul}
 -\nabla\cdot \epsilon \nabla \overline{u}_{\mbox{\footnotesize long}} + 
 \kappa^2 \overline{u}_{\mbox{\footnotesize long}}=\rho_{\mbox{\footnotesize long}} 
 \quad \mbox{in } \quad \Omega,
\end{equation}
and equipped with the same interface and boundary conditions as the initial PBE.
Indeed, since $u_{\mbox{\footnotesize short}}$ is localized within each sphere $S_k$ 
of a small radius $\sigma$ centers at $x_k$ (i.e., both $u_{\mbox{\footnotesize short}}$ 
and its co-normal derivative vanish on $\Gamma$), we deduce that 
$\overline{u}_{\mbox{\footnotesize long}} = u - u_{\mbox{\footnotesize short}}$ inherits the same
interface conditions on $\Gamma$ as the solution $u$ of PBE. 
 
 The discrete version of the above regularization scheme is discussed in 
  \cite{BKhor_Dirac:18}, \S3.4,  while the numerical tests for some moderate 
  size proteins have been conducted in \cite{BeKhKhKwSt:18}. 
  The extension of this regularization scheme to the case of nonlinear PBE 
  is described in \cite{BKKKM_PBE:18}.
  
  From Lemma \ref{lem:InterEnergy} we know that the short range part in the tensor representation 
  of collective potential does not contribute into the free-space electrostatic energy.
    Combining this with the fact that our regularization scheme effects only the data 
  in the interior of   $\Omega_m$, where the equation coefficient is constant, 
  it can be shown that the energy 
  of the regularized solution is the same as that for the solution of the initial PBE. 
  This explains why we call the  
  presented regularization scheme as {\it energy preserving}, i.e., {\it conservative}.

 \section{Conclusions}\label{sec:Conclusions}
 
 In this paper,  we outline the prospects for tensor-based numerical modeling of the
collective electrostatic  potentials on lattices and in many-particle 
systems of general type. 
We introduce the low-rank tensor approximation method for calculation of interaction potential
for  many-body  systems with variable charges 
 placed on $L^{\otimes d}$ lattices and discretized on fine $n^{\otimes d}$ Cartesian grids.
As result, the interaction potential is represented 
in a parametric low-rank canonical format in $O(d L n)$ complexity, while 
the energy is then calculated in $O(d L)$ operations.

Electrostatics in large biomolecules can be modeled by using the novel 
range-separated tensor format, 
which maintains the long-range part of  the free space  3D collective 
potential of many-particle system by using a parametric low-rank format in $O(n)$-complexity. 
 We demonstrate the RS decomposition for the Slater function and present the 
comparative analysis with the case of Newton kernel.
We show how the force field can be easily recovered by using the already precomputed electric field 
in the low-rank RS format.
We demonstrate that the RS tensor representation of the discretized Dirac delta 
enables the construction of the efficient energy preserving regularization scheme
for solving the 3D elliptic PDEs with strongly singular right-hand side
arising, in particular, in  bio-molecular  modeling.


 The presented techniques  demonstrate that 
the tensor-based approximation methods suggest the  powerful numerical tools 
 which can be applied for many-body dynamics, protein docking and classification problems, 
for low-parametric interpolation of big data, as well as in machine learning.

 \section{Appendix: Sketch of the rank-structured tensor formats}\label{sec:Append}
 
The tensor decompositions have been since long used in chemometrics, 
psychometrics, and signal processing for the quantitative analysis of the experimental data,
without special demands on accuracy and data size 
\cite{De_Lath_PhD:97,Comon:02,Cichocki:2002,smilde-book-2004}.
The basic rank-structured  representations  of tensors in multilinear algebra 
are the canonical \cite{Hitch:27} and Tucker \cite{Tuck:1966} tensor formats. 

We consider tensors of order $d$ as the multi-indexed data array  
\begin{equation}
 \label{Tensor_def}
{\bf T}=[t_{i_1,\ldots,i_d}]  \in 
\mathbb{R}^{n_1 \times \ldots \times n_d}\;
\mbox{  with } \quad i_\ell\in {\cal I}_\ell:=\{1,\ldots,n_\ell\}.
\end{equation}
Multilinear operations with such tensors scale exponentially in the dimension parameter, 
as $O(n^d)$ (assuming $n_\ell =n$).

  A tensor in the $R$-term canonical format is defined by a sum of rank-$1$ tensors  
 \begin{equation}\label{eqn:CP_form}
   {\bf T} = {\sum}_{k =1}^{R} \xi_k
   {\bf u}_k^{(1)}  \otimes \ldots \otimes {\bf u}_k^{(d)},  \quad  \xi_k \in \mathbb{R},
\end{equation}
where ${\bf u}_k^{(\ell)}\in \mathbb{R}^{n_\ell}$ are normalized vectors, 
and $R$ is the canonical rank. The storage cost of this
parametrization is bounded by  $d R n$, which is linear in  both  $n$ and $d$. 
Thus, the canonical tensor format
avoids the curse of dimensionality when increasing the dimension parameter $d$.
Though,  there are no stable algorithms for the canonical-type 
approximate representation of full format tensors. 

The Tucker tensor format is suitable for stable numerical decomposition of tensors with a fixed
truncation threshold.
We say that the tensor ${\bf T} $ is represented in the rank-$\bf r$ orthogonal Tucker format 
with the rank parameter ${\bf r}=(r_1,\ldots,r_d)$ if 
\begin{equation}
\label{eqn:Tucker_form}
  {\bf T}  =\sum\limits_{\nu_1 =1}^{r_1}\ldots
\sum\limits^{r_d}_{{\nu_d}=1}  \beta_{\nu_1, \ldots ,\nu_d}
\,  {\bf v}^{(1)}_{\nu_1} \otimes  
{\bf v}^{(2)}_{\nu_2} \ldots \otimes {\bf v}^{(d)}_{\nu_d}, 
\end{equation}
where $\{{\bf v}^{(\ell)}_{\nu_\ell}\}_{\nu_\ell=1}^{r_\ell}\in \mathbb{R}^{n_\ell}$, $\ell=1,\ldots,d$
represents a set of orthonormal vectors 
and $\boldsymbol{\beta}=[\beta_{\nu_1,\ldots,\nu_d}] \in \mathbb{R}^{r_1\times \cdots \times r_d}$ is 
the Tucker core tensor. The exponential growth of the storage size is not avoided for the Tucker type
decomposition, but is reduced to the core size, $r^d$ ($r=\max r_\ell$),
where in the usual practice $r\ll n$.  Hence, this format is suited for applications in moderate 
dimensions, say, for problems in $\mathbb{R}^3$.

The standard Tucker tensor approximation method is based on 
the higher order singular value decomposition (HOSVD) \cite{DMV-SIAM2:00,De_Lath_PhD:97}
of the complexity $O(n^{d+1})$.
This techniques require the  target  tensor in full size format,
hence, it becomes non-tractable in numerical analysis of tensors 
with large mode size  or/and large $d$  due to storage limitations. 
The HOSVD was further extended to the TT/HT tensor formats 
\cite{OsTy_TT:09,Gras:2010,Osel_TT:11}. 


  Tensor numerical methods  in scientific computing  emerged when 
it was found that the  Tucker/canonical  tensor decompositions for 
function related tensors exhibit exceptional approximation properties. 
It was proven and demonstrated numerically  \cite{Khor1:06,KhKh:06} that 
for tensors arising from the discretization of some classes of multivariate operators and 
functions in $\mathbb{R}^d$ the approximation error of the Tucker decomposition 
decays exponentially in the Tucker rank. 
Previous papers on the low-rank approximation of the multidimensional 
functions and operators, in particular, based on sinc-quadratures 
\cite{GHK6:03,GHK:05,HaKhtens:04I}, described the constructive way
for the analytical low-rank canonical representations. 

The application of tensor numerical methods in computational quantum chemistry was an important 
motivating step \cite{khor-ml-2009,KhKhFl_Hart:09}.
It was shown that calculation of the three-dimensional convolution operators in the 
Hartree-Fock equation 
can be reduced to operations which scale linearly in the univariate grid size $n$.
 In fact, both the direct and the assembled tensor summation methods for electrostatic 
potentials (see section \ref{ssec:direct} and \cite{KhKh_CPC:13}) were
initiated by application of the tensor-based Hartree-Fock solver \cite{Khor_bookQC_2018} to 
 compact molecules and lattice-structured molecular systems, respectively. 
    
   
An important ingredient  of the tensor numerical methods was the canonical-to-Tucker 
decomposition for tensors of type (\ref{eqn:CP_form}) with large initial rank 
based on the reduced higher order singular value decomposition (RHOSVD)
introduced in \cite{khor-ml-2009}. 
Its computational complexity scales linearly in the dimension size $d$, $O(d n^2 R)$,
since it does not require the construction of a full size tensor.
The RHOSVD is  the kernel of  of the basic rank reduction schemes in 
tensor computations in scientific computing \cite{Khor_bookQC_2018}, 
 in particular, for  the construction of the 
RS tensor decompositions in the numerical treatment of electrostatics in many particle systems.

The most considerable contribution to tensor numerical methods in many dimensions is the invention of 
the tensor-train (TT) 
format in \cite{OsTy_TT:09,Osel_TT:11}, which provides computations on  
multidimensional data arrays with linear complexity scaling in both dimension $d$ and 
the univariate grid size $n$. 
This format is defined as follows: Given the rank parameter 
${\bf r}=(r_1,...,r_d)$, $r_d=1$,  the  tensor
${\bf A}=[a_{\bf i}] \in \mathbb{R}^{n^{\otimes d}}$ belongs to
the rank-${\bf r}$ TT format if it can be parametrized 
by contracted product of tri-tensors in 
$\mathbb{R}^{r_{\ell-1}\times n_\ell \times r_\ell} $,
\begin{eqnarray*}
a_{i_1...i_d}& = 
& {\sum}_{\boldsymbol{\alpha}} 
G^{(1)}_{ \alpha_1}[{i_1}]
G^{(2)}_{\alpha_1\alpha_2}[i_2] \cdots G^{(d)}_{\alpha_{d-1}}[i_d] 
 \equiv  {G^{(1)}[{i_1}]} {G^{(2)}[{i_2}]}... {G^{(d)}[{i_d}]},
\label{Def:TT}
\end{eqnarray*}
where ${G^{(\ell)}[{i_\ell}}]$ is an 
$r_{\ell-1}\times r_{\ell}$ matrix for $1\leq i_\ell \leq n_\ell$.
This explains the alternative name the matrix product state (MPS) traditionally used in 
quantum chemistry computations \cite{White:93,Scholl:11}. The storage size for this product type parametrization 
is estimated by $d r^2 n$.
The canonical, Tucker, TT and hierarchical Tucker (HT) \cite{Hack_Book12} tensor formats 
proved to provide the efficient numerical 
tools in various applications 
\cite{OsTy_TT2:09,DoKhOsel:11,OsDo:11,lubich-dynht-2012,LuOsVa:15,SulimTyrt:2017,BeOnSt:19,HaKrUs:17}. 

The quantized tensor train (QTT) decomposition introduced in \cite{KhQuant:09,Osel_QTT:10}  
allows to further reduce the representation complexity  for multidimensional tensors  
to logarithmic scale in the volume size,  $O(d \log n)$. 
 QTT  tensor decompositions are  nowadays      
widely used in solution of the multidimensional problems in numerical analysis,  see for example, 
\cite{RaOs_Conv:15,DoKhOsel:11,KazKhamNipSchw:13,KaReSch:13,LuOsVa:15,Cichocki:2016}, and 
 \cite{KressSteinUsch:13,DolKhLitMat:14,RaOs:16,MaEsLiHaZa:19,BaSchUs:16,MaRaSch:19,MaEsLiHaZa:19},
 as well as the  recent books on tensor numerical methods  in
computational quantum chemistry and in scientific computing \cite{Khor_bookQC_2018,Khor-book-2018} 
and references therein.

\begin{footnotesize}

\end{footnotesize}

\end{document}